\documentclass[11pt,a4paper,reqno]{amsart}

\usepackage[polish,english]{babel}

\usepackage[margin=1in]{geometry}
\usepackage{fancyhdr}
\usepackage[utf8]{inputenc}

\usepackage{amsmath,amsfonts,amsbsy,amsgen,amscd,mathrsfs,amssymb,amsthm,mathtools}
\usepackage[foot]{amsaddr}
\usepackage{enumitem}
\usepackage{thmtools} 
\usepackage{thm-restate}
\usepackage{dsfont}
\usepackage[T1,T2A]{fontenc}
\RequirePackage[mathscr]{euscript}
\DeclareSymbolFont{rsfs}{U}{rsfs}{m}{n}
\DeclareSymbolFontAlphabet{\mathscrsfs}{rsfs}

\usepackage{algorithm,algorithmic}
\usepackage{graphicx}
\usepackage{tikz}
\usepackage{pgfplots}
\pgfplotsset{compat=1.15}
\usepackage{mathrsfs}
\usetikzlibrary{arrows}

\usepackage[color=yellow]{todonotes}

\usepackage{xcolor}
\definecolor{CarmineRed}{rgb}{0.585938, 0, 0.09375}
\definecolor{UltraMarine}{rgb}{0.0703125, 0.0390625, 0.558594}

\usepackage{hyperref}
\hypersetup{
    colorlinks,
    linkcolor= {CarmineRed}, 
    citecolor= {UltraMarine},
    urlcolor= {UltraMarine} 
}

\usepackage{cleveref}
\usepackage[T1]{fontenc} % added for \k{e} 

\makeatletter
\def\subsubsection{\@startsection{subsubsection}{3}%
\z@{.5\linespacing\@plus.7\linespacing}{-.5em}%
{\normalfont\itshape}}
\makeatother

\newcommand{\D}{{\mathop{}\!\mathrm{d}}} 
\newcommand{\R}{\mathbb{R}}
\newcommand{\N}{\mathbb{N}}

\newcommand{\A}{\mathbb{A}}
\newcommand{\bS}{\mathbb{S}}
\newcommand{\X}{\mathbb{X}}
\newcommand{\KL}{D_{\operatorname{KL}}}
\newcommand{\cD}{\mathscrsfs{D}}

\DeclareMathOperator*{\TV}{\operatorname{TV}}
\DeclareMathOperator*{\FR}{\operatorname{FR}}
\DeclareMathOperator*{\argmax}{arg\,max}
\DeclareMathOperator*{\argmin}{arg\,min}

\numberwithin{equation}{section}  

\newtheorem{definition}{Definition}[section]
\newtheorem{example}[definition]{Example}
\newtheorem{remark}[definition]{Remark}
\newtheorem{theorem}[definition]{Theorem}
\newtheorem{proposition}[definition]{Proposition}
\newtheorem{corollary}[definition]{Corollary}
\newtheorem{lemma}[definition]{Lemma}

\newtheorem{setting}[definition]{Setting}

\newtheorem{asu}[definition]{Assumption}

\title[Fisher-Rao Gradient Flows of Linear Programs]{
Fisher-Rao Gradient Flows of Linear Programs and State-Action Natural Policy Gradients
} 

\author[M\"{u}ller]{Johannes M\"{u}ller$^{1,\heartsuit}$}
\email{mueller@mathc.rwth-aachen.de}
\author[\c{C}ayc{\i}]{Semih \c{C}ayc{\i}$^1$}
\email{cayci@mathc.rwth-aachen.de}
\author[Mont{\'u}far]{Guido Mont{\'u}far$^{2,3}$}
\email{montufar@math.ucla.edu}

\address{$^1$ Department of Mathematics, RWTH Aachen University, Aachen, 52062, Germany}
\address{$^2$ Departments of Mathematics and Statistics \& Data Science, University of California, Los Angeles, 90095, USA}
\address{$^3$ Max Planck Institute for Mathematics in the Sciences, Leipzig, 04103, Germany}
\address{$^\heartsuit\!$ Corresponding author}

\begin{document}

\fancyhead[CE]{Fisher-Rao Gradient Flows of Linear Programs}
\fancyhead[CO]{M\"{u}ller, \c{C}ayci, Mont\'{u}far}

\maketitle

\begin{abstract}
Kakade's natural policy gradient method has been studied extensively in the last years showing linear convergence with and without regularization. 
We study another natural gradient method which is based on the Fisher information matrix of the state-action distributions and has received little attention from the theoretical side. 
Here, the state-action distributions follow the Fisher-Rao gradient flow inside the state-action polytope with respect to a linear potential. 
Therefore, we study Fisher-Rao gradient flows of linear programs more generally and show linear convergence with a rate that depends on the geometry of the linear program. 
Equivalently, this yields an estimate on the error induced by entropic regularization of the linear program which improves existing results. 
We extend these results and show sublinear convergence for perturbed Fisher-Rao gradient flows and natural gradient flows up to an approximation error. 
In particular, these general results cover the case of state-action natural policy gradients. 
\\ \textbf{Keywords: }{
Fisher-Rao metric, linear program, entropic regularization, multi-player game, Markov decision process, natural policy gradient
}
\\ \textbf{MSC codes: }{
65K05, 90C05, 90C08, 90C40, 90C53
}
\end{abstract}

\section{Introduction}

Natural policy gradient (NPG) methods and their proximal and trust region formulations known as PPO and TRPO are among the most popular policy optimization techniques in modern reinforcement learning (RL). 
As such they serve as a cornerstone of many recent RL success stories including celebrated advancements in computer games~\cite{schulman2015trust, schulman2017proximal, berner2019dota} and the recent development of large language models like ChatGPT~\cite{achiam2023gpt}. 
This has motivated a quickly growing body of work studying the theoretical aspects such as the convergence properties and statistical efficacy of natural policy gradient methods. 
Almost all of these works consider a specific model geometry where the Fisher-Rao metrics of the individual rows of the policy are mixed according to their state distribution or slight modifications of this~\cite{kakade2001natural, bagnell2003covariant,montufar2014fisher, kerimkulov2023fisher}. 
However, other choices for the model geometry are possible. In particular, the Fisher metric on the state-action distributions has been used to design a natural gradient method as well as actor-critic and a trust-region variant known as relative entropy search (REPS)~\cite{morimura2008new, morimura2009generalized, peters2010relative}. 
This alternative natural policy gradient has been found to have the potential to reduce the severity of plateaus~\cite{morimura2008new} and improve the performance of actor-critic methods~\cite{morimura2009generalized}. 
Despite these findings, theoretical results remain scarce. 
Initial works on the convergence of state-action natural policy gradients show an exponential convergence guarantee~\cite{muller2023geometry}, without quantifying the exponential rate and without addressing function approximation. 
In this article, we provide quantitative convergence results for state-action natural policy gradients both with and without function approximation. 

For our theoretical analysis, we work in the space of state-action distributions which brings the benefit that the reward optimization problem becomes a linear program~\cite{kallenberg1994survey}. 
In particular, for rich enough parametric policy models, the state-action natural policy gradient methods can be described by the Fisher-Rao gradient flow of the state-action linear program~\cite{muller2023geometry}. 
This motivates us to study Fisher-Rao gradient flows for general linear programs. 
These flows coincide with the solutions of entropy-regularized linear programs and thus by studying the convergence of the flow we also bound the error introduced by entropic regularization in linear programming. 

\subsection{Contributions}
It is the goal of this article to provide insights into the convergence properties of state-action natural policy gradients with and without function approximation. 
To this end, we first provide an explicit convergence analysis of Fisher-Rao gradient flows of general linear programs. 
More precisely, our main contributions can be summarized as follows:
\begin{itemize}
    \item We study Fisher-Rao flows of general linear programs. 
    Leveraging a local generalized strong convexity condition we show linear convergence both in KL-divergence and function value with an exponential rate depending on the geometry of the linear program, see \Cref{thm:convergenceFRGF} and \Cref{thm:convergence-without-uniqueness} for unique and non-unique optimizers. 
    \item We obtain an estimate on the regularization error in entropy regularized linear programming improving known convergence rates, 
    see \Cref{cor:regularizationError}. 
    \item We study natural gradients for parametric measures and show sublinear convergence under inexact gradient evaluations up to an approximation error and a distribution mismatch measured in the $\chi^2$-divergence, see \Cref{cor:convergenceGeneral}. 
    \item In a multi-player game with a specific payoff structure, we show linear convergence of the natural gradient flow, see \Cref{thm:multiPlayer}. 
    \item In the context of Markov decision processes, we study state-action natural policy gradients and  
    provide a sublinear convergence result for general policy parametrizations, see \Cref{cor:sublinearConvergenceNPG}, and a linear convergence guarantee gradient for regular parametrizations, see \Cref{thm:linearConvergenceTabular}. In particular, this covers tabular softmax, escort, and log-linear parameterizations. 
    \item For non-unique optimizers, the asymptotic limit of Fisher-Rao gradient flows is known to be the information projection of the initial condition to the set of optimizers. 
    We strengthen this by providing an exponential convergence rate, see \Cref{thm:convergence-without-uniqueness}, and by extending this result to state-action natural policy gradients. 
    This shows that state-action natural gradients converge to an optimal policy that achieves maximal entropy over states and actions, which characterizes its \emph{implicit bias}, see \Cref{thm:convergence-without-uniqueness} and \Cref{thm:linearConvergenceTabular}. 
\end{itemize}

\subsection{Related works} 
State-action natural policy gradients were recently studied with and without state-action entropy regularization in~\cite{muller2023geometry}. 
For regularization strength $\lambda>0$ that work showed $O(e^{-\lambda t})$ convergence, but in the unregularized case, the precise exponential rate was not characterized. 

A mirror descent variant of the state-action natural policy gradients was shown to achieve an optimal $O(\sqrt{T})$ regret in an online setting in~\cite{zimin2013online, dick2014online, neu2017unified}. 

There has been a recent surge of works studying the natural policy gradient method proposed by  Kakade. 
The initial results of~\cite{agarwal2021theory}
showed sublinear convergence rate $O(t^{-1})$ for unregularized problems. This was subsequently improved to a linear rate for step sizes found by exact line search~\cite{bhandari2021linear} and constant step sizes~\cite{khodadadian2022linear, alfano2022linear, yuan2022linear}. 
    For regularized problems, the method converges linearly for small step sizes, locally quadratically for Newton-like step sizes, and linearly with linear function approximation~\cite{cen2021fast,li2023quasi}. 
    The linear convergence of NPG has been extended to the function approximation regime and more general problem geometries, where these results either require geometrically increasing step sizes~\cite{xiao2022PGjmlr,alfano2022linear, yuan2022linear, alfano2023novel} or entropy regularization~\cite{cayci2024convergence, lan2022policy, zhan2021policy, li2023quasi,alfano2023novel}.
    However, these geometries do not cover the state-action geometries. 
    Apart from the works on convergence rates for policy gradient methods for standard MDPs, a primal-dual NPG method with sublinear global convergence guarantees has been proposed for constrained MDPs~\cite{ding2020natural, ding2022convergence}. 
Where all of these results work in discrete time, the gradient flows corresponding to this type of natural policy gradient have been shown to converge linearly under entropy regularization for Polish state and action spaces~\cite{kerimkulov2023fisher}. 

Hessian geometries, which provide a rich generalization of the Fisher-Rao metric, have been studied in convex optimization both from a continuous time perspective and via a discrete-time mirror descent analysis~\cite{alvarez2004hessian, wang2022hessian}. 
In the context of linear programming, linear convergence of the Fisher-Rao gradient flow was shown in~\cite{alvarez2004hessian} albeit without a characterization of the convergence rate. 

In the case of a linear program, the Fisher-Rao gradient flow parametrized by time corresponds to the trajectory of solutions of the entropy-regularized program parametrized by the inverse regularization strength, which has been studied in several works. 
An exponential convergence result was obtained in~\cite{cominetti1994asymptotic} and subsequently, the rate was characterized as $O(e^{-\delta t})$ for a constant $\delta$ depending on the linear program~\cite{weed2018explicit, suarez2023perspectives}. 
The results obtained in this article follow an alternative proof strategy and provide exponential convergence $O(e^{-\Delta t})$, where $\Delta\ge\delta$, where we show that the improvement can be arbitrarily large, see \Cref{ex:improvement}. 
This improvement can be strict for the linear programs encountered in Markov decision processes under standard assumptions. 
Whereas existing works study convergence in function value, our results also cover convergence in the KL-divergence.  
Finally, the geometry of Fisher-Rao gradient flows or equivalently the entropic central path was recently described as the intersection of the feasible region with a toric variety~\cite{sturmfels2024toric}.

\subsection{Notation and terminology} 
For a finite set $\X$, we denote the \emph{free vector space} over $\X$ by $\R^\X = \{ \mu\colon \X \to \R \}$. 
Its elements can be identified with vectors $(\mu_x)_{x\in\X}$. 
Similarly, we denote the vectors with non-negative entries and positive entries by $\R_{\ge0}^\X$ and $\R_{>0}^\X$, respectively. 
For two elements $\mu, \nu\in\R^\X$ we denote the \emph{Hadamard product}, i.e., the entrywise product, between $\mu$ and $\nu$ by $\mu\odot\nu\in\R^\X$, so that $\mu\odot\nu(x) \coloneqq \mu(x)\nu(x)$.
The \emph{total variation} norm $\lVert \cdot \rVert_{\TV}\colon\R^\X\to\R$ is given by $\lVert \mu \rVert_{\TV} \coloneqq \frac12\sum_{x}\lvert \mu_x \rvert$. 
Finally, with $\mathds{1}_\X\in\R^\X$ we denote the all-one vector. 

A \emph{polyhedron} is a set $P = \{ \mu\in\R^\X : \ell_i(\mu) \ge 0 \text{ for } i=1, \dots, k \}\subseteq \R^\X$, where $\ell_i\colon\R^\X\to\R$ are affine linear functions for $i=1, \dots, k$. 
A bounded (and thus compact) polyhedron is called a \emph{polytope}. 
A polytope can be shown to be the convex hull of finitely many extreme points, which are called \emph{vertices} and which we denote by $\operatorname{Vert}(P)$. 
Two vertices $\mu_1, \mu_2\in\operatorname{Vert}(P)$ are called \emph{neighbors} if the subspace $\{ c\in\R^\X : c^\top \mu_1 = c^\top \mu_2 = \max_{\mu\in P} c^\top \mu \}$ has dimension $\lvert \X \rvert-1$. 
We denote the set of all neighbors of a vertex $\mu$ by $N(\mu) \subseteq\operatorname{Vert}(P)$.  
The \emph{affine space} $\operatorname{aff\,span}(P)$ of a polytope $P\subseteq\R^\X$ is the smallest affine subspace of $\R^\X$ containing $P$. 
The relative interior $\operatorname{int}(P)$ and boundary $\partial P$ of $P$ are the interior and boundary of $P$ in its affine hull. 
Finally, the \emph{tangent space} $TP$ of $P$ is given by the linear part of $\operatorname{aff\,span}(P)$. 

We call $\Delta_\X \coloneqq \left\{ \mu\in\R^\X_{\ge0} : \sum_{x} \mu_x=1 \right\}$ the \emph{probability simplex}. 
We say that $\mu\in\Delta_\X$ is absolutely continuous with respect to $\nu\in\Delta_\X$ if $\nu(x) = 0$ implies $\mu(x)=0$ and write $\mu \ll \nu$. 
We denote the expectation with respect to $\mu\in\Delta_\X$ by $\mathbb E_\mu$ and call $\chi^2(\mu, \nu)\coloneqq \mathbb E_\nu\left[\frac{(\mu(x)-\nu(x))^2}{\nu(x)^2}\right]$ the \emph{$\chi^2$-divergence} between $\mu$ and $\nu$. 
If $\mathbb Y$ is another finite set, we call the Cartesian product $\Delta_\X^\mathbb Y = \Delta_\X \cdot\ldots\cdot \Delta_\X$ the \emph{conditional probability polytope} and associate its elements with stochastic matrices $P\in\R^{\X\times\mathcal Y}_{\ge0}$ with $\sum_x P(x|y) = 1$. 

For a differentiable function $f\colon\Omega\to \R$ on an open subset $\Omega\subseteq\R^\X$ we denote the Euclidean gradient and Hessian of $f$ at $\mu\in\R^\X$ by $\nabla f(\mu)\in\R^{\X}$ and $\nabla^2 f(\mu)\in\R^{\X\times\X}$. 

Finally, for a differentiable curve $(c_t)_{t\in I}\subseteq\mathcal{M}$ defined on an interval $I\subseteq\R$ mapping to a manifold $\mathcal{M}$ we denote its time derivative by $\partial_t c_t$. 

\section{Preliminaries on Fisher-Rao Gradient Flows}
To gain insight into natural gradient descent methods, we study their time-continuous version which is given by $\partial_t\theta_t = -F(\theta_t)^+ \nabla f(\mu_{\theta_t})$, where $\mu_\theta$ is a parametrized measure model and $f(\mu)$ is an objective function and $F(\theta)$ is the Fisher-information matrix~\cite{amari1998natural}. 
The objective function can be a log-likelihood in the case of maximum likelihood estimation or a linear function in the case of reinforcement learning as we will see in \Cref{sec:state-action-NPG}. 
The Fisher-information matrix is closely connected to a specific Riemannian geometry, the Fisher-Rao metric, on the space of probability measures, which we introduce and discuss here. 
As we study gradient-based optimizers, we put a special emphasis on gradient flows with respect to the Fisher-Rao metric and provide a self-contained review of the properties of Fisher-Rao gradient flows that we require later. 
The results in this section can be generalized to a large class of Hessian geometries and -- apart from the central path property -- also to other objectives albeit with different proofs, for which we refer to~\cite{alvarez2004hessian, mueller2023thesis}. 

The \emph{Fisher-Rao metric} is a Riemannian metric on the positive orthant given by 
\begin{equation}
    g_{\mu}^{\FR}(v,w) \coloneqq \sum_{x\in \X} \frac{v_x w_x}{\mu_x} \quad \text{for all } v, w\in\R^{\X}, \mu\in\R^{\X}_{>0},
\end{equation}
where we denote the induced norm by $\lVert v \rVert_{g_\mu^{\FR}} \coloneqq {g_{\mu}^{\FR}(v,v)}^{\frac12}$. 
The Fisher-Rao metric was introduced in the seminal works of C.\ R.\ Rao~\cite{radhakrishna1945information, rao1987differential} to provide lower bounds on the statistical error in parameter estimation known as the Cramer-Rao bound. 
This geometric approach to statistical estimation has subsequently led to the development of the field of information geometry, where N.\ N.\ \v{C}encov characterized the Fisher-Rao metric as the unique Riemannian metric (up to scaling) that is invariant under sufficient statistics~\cite{vcencov1978algebraic, amari2016information, ay2017information}. 
Despite its central role in statistics, our main motivation for studying the Fisher-Rao metric is for its use in reinforcement learning, where it has been used to design natural gradient algorithms as well as trust region methods~\cite{amari1998natural, morimura2008new, peters2010relative}. 
Further, it is very closely related to entropic regularization in linear programming, which enjoys immense popularity, particularly in computational optimal transport~\cite{peyre2019computational, suarez2023perspectives}, see also~\cite{weed2018explicit} for a detailed discussion of entropy regularized linear programming. 

The Fisher-Rao metric is closely connected to the negative {Shannon entropy} 
\begin{equation}
    \phi(\mu) = -H(\mu) \coloneqq \sum_{x\in\X} \mu_x \log \mu_x  \quad \text{for all } \mu\in\R^{\X}_{>0}
\end{equation}
as it is induced by the Hessian of the (negative) entropy, meaning that we have $g_{\mu}^{\FR}(v,w) = v^\top \nabla^2 \phi(\mu) w$ for all $v, w\in\R^{\X}, \mu\in\R^{\X}_{>0}$. 
As such, the Fisher-Rao metric falls into the class of \emph{Hessian metrics} that have been studied in convex optimization; we refer to~\cite{alvarez2004hessian, mueller2023thesis} for general well-posedness and convergence results. 
An important concept in the analysis of Hessian gradient flows is the \emph{Bregman divergence} induced by $\phi$, which in the case of the negative entropy is given by the \emph{KL-divergence} 
\begin{equation}
    \KL (\mu, \nu) \coloneqq  \phi(\mu) - \phi(\nu) - \nabla \phi(\nu)(\mu-\nu) = \sum_{x\in \X} \mu_x \log \frac{\mu_x}{\nu_x} - \sum_{x\in \X} \mu_x + \sum_{x\in \X} \nu_x 
\end{equation}
for $\mu, \nu\in\R_{\ge0}^\X$ with $\mu\ll\nu$, where we use the common convention $0\log\frac00\coloneqq0$. 

Consider now a continuously differentiable function $f\colon\R_{\ge0}^\X\to\R$ that we assume to be differentiable on $\R_{>0}^\X$ that we want to optimize over a polytope $P= \R_{\ge0}^{\X}\cap\mathcal L$, where $\mathcal L$ is a linear space. 
We denote the gradient of $f\colon \R_{>0}^\X\to\R$ at $\mu\in\R_{>0}^\X$ with respect to the Fisher-Rao metric by $\nabla^{\FR}f(\mu)$ and call it the \emph{Fisher-Rao gradient}. 
Further, we denote the Fisher-Rao gradient of $f\colon\operatorname{int}(P)\to\R$ by $\nabla^{\FR}_P f(\mu)\in TP$, which is uniquely determined by 
\begin{equation}
    g_\mu^{\FR}(\nabla^{\FR}_P f(\mu),v) = df(\mu)v \quad \text{for all } v\in TP.
\end{equation}
Note that $\nabla^{\FR}_P f(\mu)$ is the projection of $\nabla^{\FR}f(\mu)$ with respect to the Fisher-Rao metric onto $TP$.
By examining the definition of the Fisher-Rao metric we see that this is equivalent to 
\begin{equation}
\label{eq:characterizationFRGradient}
    \langle \nabla^2\phi(\mu)\nabla^{\FR}_P  f(\mu), v \rangle = \langle \nabla f(\mu), v \rangle \quad \text{for all } v\in TP. 
\end{equation}
    
We say that $(\mu_t)_{t\in [0, T)} \subseteq \operatorname{int}(P)$ solves the \emph{Fisher-Rao gradient flow} if it solves the gradient flow with respect to the Fisher-Rao metric, i.e., if 
\begin{equation}\label{eq:FRGF}
    \partial_t \mu_t = \nabla^{\FR}_P  f(\mu_t) \quad \text{for all } t\in[0, T). 
\end{equation}
By using the characterization~\eqref{eq:characterizationFRGradient} of $\nabla^{\FR}_P f(\mu_t)$, we see that $(\mu_t)_{t\in [0, T)} \subseteq \operatorname{int}(P)$ solves the Fisher-Rao gradient flow~\eqref{eq:FRGF}
if and only if we have 
\begin{equation}\label{eq:explicitFRGF}
    \langle \nabla^2 \phi(\mu_t) \partial_t \mu_t, v \rangle = \langle \nabla f(\mu_t), v \rangle 
    \quad \text{for all } v\in TP, t\in[0, T). 
\end{equation}

In the remainder, we study linear programs and work in the following setting. 

\begin{setting}\label{setting:generalLP}
We consider a finite set $\X$ and a linear program 
    \begin{equation}\label{eq:LP}
        \max c^\top \mu \quad \text{subject to } \mu\in P, 
    \end{equation}
with cost $c\in\R^{\X}$ and feasible region $P = \R^\X_{\ge0} \cap  \mathcal L$ with $P \cap \R^\X_{>0}\ne\emptyset$, where $\mathcal L\subseteq\R^{\X}$ is an affine space. 
By $(\mu_t)_{t\in[0, T)}\subseteq\operatorname{int}(P)$ we denote a solution of the Fisher-Rao gradient flow~\eqref{eq:FRGF} with initial condition $\mu_0\in P \cap \R_{>0}^\X$ and potential $f(\mu) = c^\top \mu$, where $T\in\R_{\ge0}\cup\{+\infty\}$. 
\end{setting}

Fisher-Rao gradient flows are closely connected to the solutions of KL-regularized linear programs, $c^\top \mu - \lambda \KL (\mu, \mu_0)$. 
The family of solutions of the regularized problems parametrized by the regularization strength $\lambda$ is referred to as the (entropic) \emph{central path} in optimization~\cite{boyd2004convex}.  

\begin{proposition}[Central path property,\cite{alvarez2004hessian}]
    Consider Setting~\ref{setting:generalLP}. 
    Then $\mu_t$ is uniquely characterized by 
    \begin{equation}\label{eq:centralPath}
        \mu_t = \argmax\left\{ c^\top \mu - t^{-1} \KL (\mu, \mu_0) : \mu \in P \right\} \quad \text{for all } t\in(0,T). 
    \end{equation}
\end{proposition}
    \begin{proof}
    Let $\hat{\mu}_t\in P$ denote the unique maximizer of $g(\mu)\coloneqq c^\top\mu - t^{-1} \KL (\mu, \mu_0)$ over $P$ for $t>0$, then surely $\hat{\mu}_t\in \operatorname{int}(P)$. 
    Thus, $\hat{\mu}_t$ is uniquely determined by $\langle \nabla g(\hat{\mu}_t), v\rangle = 0$ for all $v\in TP$. 
    Direct computation yields $\nabla g({\mu}) = c - t^{-1} (\nabla\phi(\mu) - \nabla\phi(\mu_0))$ and hence $\hat{\mu}_t$ is uniquely determined by 
    \begin{equation*}
        t\langle c, v\rangle = \langle \nabla\phi(\hat{\mu}_t) - \nabla\phi(\mu_0), v\rangle \quad \text{for all } v\in TP. 
    \end{equation*} 
    On the other hand, for the gradient flow, we can use~\eqref{eq:explicitFRGF} and compute for $v\in TP$ 
    \begin{align*}
        \langle \nabla\phi(\mu_t) - \nabla\phi(\mu_0), v \rangle & = \int_{0}^t \partial_s \langle \nabla \phi(\mu_s), v \rangle \D s 
        = \int_{0}^t \langle \nabla^2 \phi(\mu_s)\partial_s \mu_s, v \rangle \D s 
        \\ & = \int_{0}^t \langle \nabla f(\mu_s) \mu_s, v \rangle \D s 
        = \int_{0}^t \langle c, v \rangle \D s = t \langle c, v \rangle . 
    \end{align*}
    This shows $\mu_t = \hat{\mu}_t$ as claimed. 
\end{proof}

We can use the central path property to show $O(t^{-1})$ convergence. 
The following corollary can be generalized to arbitrary convex objectives~\cite{alvarez2004hessian}. 

\begin{corollary}[Sublinear convergence rate,\cite{alvarez2004hessian}]\label{prop:sublinearRate}
    Consider Setting~\ref{setting:generalLP} and assume that the linear program~\eqref{eq:LP} admits a solution $\mu^\star\in P$. 
    Then for $\mu\in P$ it holds that 
\begin{equation}\label{eq:sublinearConvergence}
        c^\top \mu^\star - c^\top\mu_t \le \frac{\KL (\mu^\star, \mu) - \KL (\mu_t, \mu)}t \le \frac{\KL (\mu^\star, \mu_0)}t \quad \text{for all } t\in[0, T). 
    \end{equation}
\end{corollary}
\begin{proof}
    We have $c^\top \mu_t - t^{-1} \KL (\mu_t, \mu) \ge c^\top \mu^\star - t^{-1} \KL (\mu^\star, \mu)$ by the central path property.
    Rearranging yields the result. 
\end{proof}

One can use the central path property to show the long-time existence of Fisher-Rao gradient flows. 
Again, the following result can be generalized to a large class of Hessian geometries and potentials $f$, see~\cite{alvarez2004hessian, mueller2023thesis}, albeit with more delicate proofs. 

\begin{theorem}[Well-posedness of  FR GFs,\cite{alvarez2004hessian}]
    Consider Setting~\ref{setting:generalLP}. 
    Then there exists a unique global solution $(\mu_{t})_{t\ge0}\subseteq\operatorname{int}(P)$ of the Fisher-Rao gradient flow~\eqref{eq:FRGF}. 
\end{theorem}
\begin{proof}
    The local existence and uniqueness follow from the Picard-Lindelöf theorem~\cite{teschl2024ordinary}. Hence, it suffices to show that the Fisher-Rao gradient flow does not hit the boundary $\partial P$ in finite time. 
    By the central path property, this is equivalent to the statement that the solutions of all KL-regularized problems~\eqref{eq:centralPath} lie in the interior $\operatorname{int}(P)$ of the polyhedron, which can be easily checked. 
\end{proof}

\section{Convergence of Fisher-Rao Gradient Flows}\label{sec:linearConvergeFRGF}
We have seen that Fisher-Rao gradient flows converge globally at a sublinear rate $O(t^{-1})$. 
We now build on this analysis and show that once the gradient flow enters a vicinity of the optimizer, it converges at a quasi-linear rate $O(t^\kappa e^{-\Delta t})$, where $\Delta>0$ depends on the geometry of the linear program and $\kappa>0$ depends on the initial condition $\mu_0$. 
Note that this yields $O(e^{-c t})$ convergence for all $c<\Delta$ and hence we also simply talk of a linear convergence rate.
We consider linear programs of the following form. 
\begin{setting}\label{setting:LP-inside-Simplex}
We consider a finite set $\X$ and a linear program 
    \begin{equation}\label{eq:LP-inside-simplex}
        \max c^\top \mu \quad \text{subject to } \mu\in P, 
    \end{equation}
with cost $c\in\R^{\X}$ and feasible region $P = \Delta_\X \cap  \mathcal L$ with $P \cap \R^\X_{>0}\ne\emptyset$, where $\mathcal L\subseteq\R^{\X}$ is an affine space. 
By $(\mu_t)_{t\ge0}\subseteq\operatorname{int}(P)$ we denote the solution of the Fisher-Rao gradient flow~\eqref{eq:FRGF} with initial condition $\mu_0\in P \cap \R_{>0}^\X$ and the potential $f(\mu) = c^\top \mu$. 
\end{setting}

The following result is the main contribution of this article, where we defer the proof to \Cref{subsec:linearConvergence}. 
We first establish it under the assumption that the linear program~\eqref{eq:LP-inside-simplex} admits a unique solution and provide a generalization in \Cref{thm:convergence-without-uniqueness}. 

\begin{restatable}[Linear convergence of Fisher-Rao GFs of LPs]{theorem}{convergenceFRGF}
\label{thm:convergenceFRGF}
Consider Setting~\ref{setting:LP-inside-Simplex} and assume that the linear program~\eqref{eq:LP-inside-simplex} admits a unique solution $\mu^\star\in P$. 
Let 
\begin{align}\label{eq:Delta}
    \Delta \coloneqq \min \left\{ \frac{c^\top \mu^\star - 
    c^\top \mu}{\lVert \mu^\star - \mu \rVert_{\TV }} : \mu\in N(\mu^\star)  \right\},
\end{align}
where $N(\mu^\star)$ denotes the set of neighboring vertices of $\mu^\star$ and set 
\begin{equation}\label{eq:definitiont_0}
    t_0 \coloneqq \frac{2 \KL (\mu^\star, \mu_0)}{\Delta \cdot \min\{ \mu^\star_x : \mu^\star_x>0 \}} . 
\end{equation}
Then for any $t\ge t_0$ we have 
\begin{align}\label{eq:linearConvergenceKL}
    \KL (\mu^\star, \mu_t) \le \KL (\mu^\star, \mu_0) \exp\left( -\Delta(t-t_0) +  2 t_0 \Delta 
    \log\left(\frac{t+t_0}{2t_0}\right) \right), 
\end{align}
as well as 
\begin{align}\label{eq:linearConvergenceValue}
    c^\top \mu^\star - c^\top \mu_t \le \Delta \KL (\mu^\star, \mu_0) \exp\left( -\Delta(t-t_0) +  2 t_0 \Delta \log\left(\frac{t+t_0}{2t_0}\right) \right). 
\end{align}
\end{restatable}

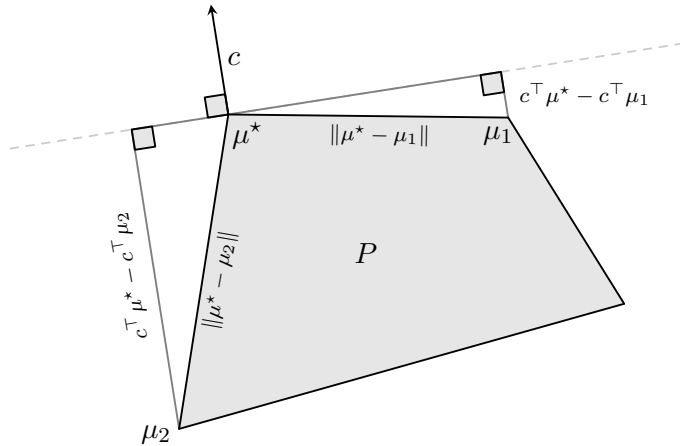
\begin{figure}
    \centering
    \definecolor{wrwrwr}{rgb}{0.7,0.7,0.7}
    \begin{tikzpicture}[line cap=round,line join=round,x=1.0cm,y=1.0cm, scale=0.6]
        \clip(-11.6,-1.3) rectangle (3.3,9.);
        \draw [line width=0.25mm,color=black,opacity=0.2,domain=-11.6:4.4,dashed] plot(\x,{(-17.320205929578584-0.3819978698484281*\x)/-2.427231763738467});
        \draw [line width=0.25mm,color=gray] (-7.896750832029924,-0.8842126265006686)-- (-8.937568621959711,5.729190744005193);
        \draw [line width=0.25mm,color=gray] (-0.6807109624408225,5.994365448617686)-- (-0.8395532985757674,7.003656845572703);
        \draw [line width=0.25mm,color=gray] (-8.937568621959711,5.729190744005193)-- (-0.8395532985757674,7.003656845572703);
        \fill[line width=0.25mm,fill=black,fill opacity=0.1] (-7.896750832029924,-0.8842126265006686) -- (-6.818002130151572,6.062768236261533) -- (-0.6807109624408225,5.994365448617686) -- (1.8625382055292548,1.8801528221170187) -- cycle;
        \draw[line width=0.25mm,fill=black,fill opacity=0.1] (-6.888382070533459,6.509965564594067) -- (-7.335579398865993,6.43958562421218) -- (-7.265199458484106,5.992388295879646) -- (-6.818002130151572,6.062768236261533) -- cycle; 
        \draw[line width=0.25mm,fill=black,fill opacity=0.1] (-8.867188681577824,5.281993415672659) -- (-8.41999135324529,5.352373356054546) -- (-8.490371293627177,5.79957068438708) -- (-8.937568621959711,5.729190744005193) -- cycle; 
        \draw[line width=0.25mm,fill=black,fill opacity=0.1] (-1.2867506269083013,6.933276905190816) -- (-1.2163706865264143,6.486079576858282) -- (-0.7691733581938804,6.556459517240169) -- (-0.8395532985757674,7.003656845572703) -- cycle; 
        \draw [line width=0.25mm] (-7.896750832029924,-0.8842126265006686)-- (-6.818002130151572,6.062768236261533);
        \draw [line width=0.25mm] (-6.818002130151572,6.062768236261533)-- (-0.6807109624408225,5.994365448617686);
        \draw [line width=0.25mm] (-0.6807109624408225,5.994365448617686)-- (1.8625382055292548,1.8801528221170187);
        \draw [line width=0.25mm] (1.8625382055292548,1.8801528221170187)-- (-7.896750832029924,-0.8842126265006686);
        \draw [-stealth,line width=0.25mm] (-6.818002130151572,6.062768236261533) -- (-7.2,8.49);
        \draw[color=black] (-6.4,5.6) node {$\mu^\star$};
        \draw[color=black] (-8.4,-1) node {$\mu_2$};
        \draw[color=black] (-0.9,5.6) node {$\mu_1$};
        \draw[color=black] (-3.8,3.) node {$P$};
        \draw[color=black] (-6.7,7.3) node {$c$};
        \begin{scriptsize}
        \node[inner sep=0pt, rotate=-1] (l1) at (-3.5,5.6) {$\lVert \mu^\star - \mu_1 \rVert$};
        \node[inner sep=0pt, rotate=80] (l1) at (-7.,2.4) {$\lVert \mu^\star - \mu_2 \rVert$};
        \node[inner sep=0pt, rotate=100] (l1) at (-9,2.5) {$c^\top\mu^\star - c^\top \mu_2$};
        \node[inner sep=0pt, rotate=0] (l1) at (1,6.5) {$c^\top\mu^\star - c^\top \mu_1$};
        \end{scriptsize}
    \end{tikzpicture}
    \caption{Visualization of the suboptimality gap $\Delta$ appearing in \Cref{thm:convergenceFRGF} associated to the linear program~\eqref{eq:LP-inside-simplex}; note that $\Delta$ deteriorates when $c$ is almost orthogonal to a face of $P$.}
    \label{fig:delta}
\end{figure}

The constant $\Delta$ depends on the geometry of the linear program, see~\Cref{fig:delta}. 
Indeed, the quotient $\frac{c^\top \mu^\star - c^\top \mu}{\lVert \mu^\star - \mu \rVert_{\TV }}$ is the slope of the objective along the edge $\mu^\star-\mu$. Consequently, $\Delta$ decreases when the cost $c$ is closer to orthogonal to a face of $P$. 

Using the central path property of Fisher-Rao gradient flows and initializing at the maximum entropy distribution in $P$ yields the following result. 

\begin{restatable}[Entropic regularization error]{corollary}{regularizationError}\label{cor:regularizationError}
    Consider Setting~\ref{setting:LP-inside-Simplex} and assume that the linear program~\eqref{eq:LP-inside-simplex} admits a unique solution $\mu^\star\in P$. 
For $t>0$ denote by $\mu^\star_t$ 
    the unique solution of the entropy-regularized linear program
    \begin{equation}\label{eq:LP:reg}
        \max c^\top \mu + t^{-1} H(\mu) \quad \text{subject to } \mu\in P , 
    \end{equation}
where $H$ denotes the Shannon entropy. 
    Then for any $t\ge t_0$ we have  
\begin{align}\label{eq:linearConvergenceKL-regularization}
        \KL (\mu^\star, \mu_t^\star) \le 
        R_H\exp\left(-\Delta(t-t_0) +  2 t_0 \Delta \log\left(\frac{t+t_0}{2t_0}\right) \right), 
    \end{align}
    as well as     \begin{align}\label{eq:linearConvergenceValue-regularization}
        c^\top \mu^\star - c^\top \mu_t^\star \le \Delta R_H \exp\left( -\Delta(t-t_0) +  2 t_0 \Delta \log\left(\frac{t+t_0}{2t_0}\right) \right), 
    \end{align}
    where $R_H\coloneqq \max_{\mu\in P} H(\mu) - \min_{\mu\in P} H(\mu) \le \log \lvert \X \rvert$ denotes the entropic radius of $P$ and 
    $\Delta>0$ and $t_0\ge0$ are defined in~\eqref{eq:Delta} and~\eqref{eq:definitiont_0}, respectively.
\end{restatable}

Similar to the convergence result, here too one can remove the uniqueness assumption, see \Cref{rem:regularizationError-withoutUniqueness}. 

\begin{remark}[Comparison with existing results]\label{rem:comparison}
In~\cite{cominetti1994asymptotic} it was shown that the regularization error for entropy-regularized linear programs decays exponentially fast, without quantifying the convergence rate. 
The convergence rate of the error, as well as that of Fisher-Rao gradient flows, was subsequently studied in~\cite{weed2018explicit, suarez2023perspectives}, 
establishing a rate $O(e^{-\delta t})$ with 
\begin{align}
    \delta \coloneqq \frac{\min \left\{ c^\top \mu^\star - 
        c^\top \mu : \mu\in \operatorname{Vert}(P)\setminus\{\mu^\star\}  \right\}}{\max \left\{\lVert \mu \rVert_{1} : \mu \in P \right\}} . 
\end{align}
For polytopes $P\subseteq\Delta_\X$ that we consider here, we have 
\begin{align*}
    \delta & = \min \left\{ c^\top \mu^\star - c^\top \mu : \mu\in \operatorname{Vert}(P)  \right\} 
        = \min \left\{ c^\top \mu^\star - c^\top \mu : \mu\in N(\mu^\star)  \right\} \le \Delta,
\end{align*}
showing that \Cref{thm:convergenceFRGF} offers an improvement of these previous results. 

For the special case $P=\Delta_\X$, for which a matching lower bound was constructed in~\cite{weed2018explicit}, the two constants agree.
More generally, it is easily checked that $\delta=\Delta$ if and only if there is a neighboring vertex $\mu\in N(\mu^\star)$ which has minimal optimality gap $c^\top \mu^\star-c^\top\mu$ and has disjoint support from $\mu^\star$. 
To see this, note that for two probability vectors  $\mu_1, \mu_2\in\Delta_\mathbb X$ we have $\lVert \mu_1 - \mu_2 \rVert_{\textup{TV}} = \frac12 \lVert \mu_1 - \mu_2 \rVert_1 \le 1$ with $\lVert \mu_1 - \mu_2 \rVert = 1$ if and only if $\mu_1$ and $\mu_2$ have disjoint support, meaning 
    \begin{align*}
        \{ x\in\mathbb X : \mu_1(x) > 0 \} \cap \{ x\in\mathbb X : \mu_2(x) > 0 \} = \emptyset. 
    \end{align*}
Hence, for $\mu\in N(\mu^\star)$ without disjoint support from $\mu^\star$ we have $\lVert \mu^\star  - \mu \rVert_{\textup{TV}}<1$. 
This implies that $\delta=\Delta$ if and only if there is a neighboring vertex $\mu\in N(\mu^\star)$ which has minimal optimality gap $c^\top \mu^\star-c^\top\mu$ and has disjoint support from $\mu^\star$.   

The constant $\Delta$ depends on the slope of $c$ along the outgoing edges and thus the local geometry of the feasible region around $\mu^\star$, where $\delta$ is simply based on the suboptimality at the neighboring vertices. Because of this, the difference between $\delta$ and $\Delta$ can be arbitrarily big as we show in \Cref{ex:improvement}.
Further, for Markov decision processes the feasible region of the (dual) linear program is a strict subset $\cD\subsetneq \Delta_{\bS\times\A}$ and under the standard exploratory \Cref{s_asu:exploration} and more than one state we have $\delta<\Delta$, see~\Cref{rem:strict-improvements}.
In \Cref{subsec:computational-example} we provide an explicit example of a Markov decision process where  $\delta < \Delta$.  

Further, for gradient flows with respect to a Riemannian metric of the form $g_\mu^\sigma(v,w) \coloneqq \sum_{x\in\X}\frac{v_x w_x}{\mu_x^\sigma}$ one can show $O(t^{-\frac{1}{\sigma-1}})$ convergence for $\sigma \in(1, 2)$, see~\cite{muller2023geometry}. 
Note that this can be extended to the case $\sigma=2$, corresponding to logarithmic barriers for which the central path converges at a $O(t^{-1})$ rate~\cite[Section 11.2]{boyd2004convex}. 
\end{remark}

\begin{example}[Arbitrarily large improvement]\label{ex:improvement}
We consider $\mathbb X = \{1,2,3,4\}$ and $\mathcal L = \{\mu\in\mathbb R^\mathbb X : \mu(1) = \alpha\}$ for $\alpha\in(0,1)$. 
Then, the vertices of $P = \Delta_\mathbb X\cap\mathcal L$ are given by $(1-\alpha)\delta_2, (1-\alpha)\delta_3$ and $(1-\alpha)\delta_4$, where $\delta_i$ denotes the Dirac at $i$. 
When choosing the cost $c=\delta_2$ we have $\delta = 1-\alpha$ but $\Delta = 1$. 
For $\alpha\nearrow1$ the rate $\delta$ deteriorates towards $0$, whereas $\Delta$ remains constant. 
The reason for this is that $\Delta$ depends on the slope of $c$ relative to the outgoing edges, whereas $\delta$  depends on the suboptimality of the neighboring vertices. 
Hence, $\delta$ can be smaller than $\Delta$ by an arbitrarily large factor. 
\end{example}

\begin{remark}[Tightness]
For $P=\Delta_\X$ we have $\mu_t(x) \sim e^{-tc_x}$ as can be seen from the first order stationarity conditions; hence, in this case, the bound is tight. 
For general $P$, in \Cref{subsec:computational-example} we provide empirical evidence that our bound on the exponent is sometimes but not always tight depending on the specific $c$. 
\end{remark}

\subsection{Convergence of Fisher-Rao Gradient Flows}\label{subsec:linearConvergence}
At the heart of the proof lies the following result, which can easily be extended to general Hessian geometries.
For this, one can follow the reasoning in~\cite[Proposition 4.9]{alvarez2004hessian}, which treats general Hessian geometries, but does not allow for time-dependent constants $\kappa_t$ and assumes the lower bound~\eqref{eq:strongConvexity} in a neighborhood of $\mu^\star$ and not only along the trajectory. 

\begin{lemma}\label{prop:ratesTrajectories} 
    Consider Setting~\ref{setting:generalLP} and assume that there is an optimizer $\mu^\star\in P$ and $\kappa_t>0$ for $t>t_0\ge0$ such that
\begin{equation}\label{eq:strongConvexity}
    c^\top \mu^\star - c^\top \mu_t \ge \kappa_t \KL (\mu^\star, \mu_t) \quad \text{for all } t> t_0. 
\end{equation}
Then we have 
\begin{equation}\label{eq:lem:KLConvergence}
    \KL (\mu^\star, \mu_t) \le \KL (\mu^\star, \mu_{0}) \exp\left( -\int_{t_0}^t \kappa_s \D s \right) \quad \text{for all } t\ge t_0,
\end{equation}
as well as 
\begin{equation}\label{eq:lem:convergenceValue}
    c^\top \mu^\star - c^\top \mu_t \le \kappa_t \KL (\mu^\star, \mu_{0}) \exp\left( -\int_{t_0}^t \kappa_s \D s \right) \quad \text{for all } t\ge t_0.
\end{equation}
\end{lemma}

For the proof of this result, we require the following identity. 
\begin{lemma}[\cite{alvarez2004hessian}]\label{lem:evolutionKL}
Consider Setting~\ref{setting:generalLP}, whereby we allow $f\colon\R_{>0}^\X\to\R$ to be an arbitrary differentiable function, and fix $\mu\in P$. 
Then for any $t\ge0$, it holds that 
    \begin{equation}\label{eq:derivativeu}
        \partial_t \KL (\mu, \mu_t) = \langle \nabla f(\mu_t), \mu_t-\mu\rangle. 
    \end{equation}
\end{lemma}
\begin{proof}
    Denoting the negative Shannon entropy by $\phi$, we compute 
    \begin{align*}
        \partial_t \KL (\mu, \mu_t) = - \partial_t \phi(\mu_t) - \partial_t \langle \nabla \phi(\mu_t), \mu - \mu_t \rangle = \langle \nabla^2 \phi(\mu_t)\partial_t \mu_t, \mu_t - \mu \rangle.
    \end{align*}
    Now~\eqref{eq:explicitFRGF} yields the claim. 
\end{proof}
\begin{proof}[Proof of \Cref{prop:ratesTrajectories}]
Using~\eqref{eq:derivativeu} and~\eqref{eq:strongConvexity} we find that for all $t\ge T$ it holds that 
$\partial_t \KL (\mu^\star, \mu_t) = c^\top \mu_t  - c^\top \mu^\star \le - \kappa_t \KL (\mu^\star, \mu_t)$. 
Now Gronwall's inequality yields~\eqref{eq:lem:KLConvergence}. 
By \Cref{prop:sublinearRate} we have for any $h>0$ that 
    \[ c^\top \mu^\star - c^\top \mu_t \le \frac{\KL (\mu^\star, \mu_{t-h})}{h} \le  \KL (\mu^\star, \mu_{0}) \cdot \frac{\exp\left(-\int_{t_0}^{t-h} \kappa_s \D s \right)}{h}. \] 
Taking the limit $h\to0$ yields~\eqref{eq:lem:convergenceValue}. 
\end{proof} 

The lower bound \eqref{eq:strongConvexity} can be interpreted as a form of strong convexity under which the objective value controls the Bregman divergence, see also~\cite{lu2018relatively, bauschke2019linear} for a discussion of gradient domination and strong convexity conditions in Bregman divergence.
To show that such a lower bound holds in the case of the linear program~\eqref{eq:LP-inside-simplex}, we first lower bound the sub-optimality gap $c^\top \mu^\star - c^\top \mu_t$ in terms of an arbitrary norm, where we will later use the total variation distance. 

\begin{restatable}{lemma}{TVLPEstimateWU}
\label{lem:TV-LP-estimate-withoutUniqueness}
Consider a polytope $P \subseteq \R^{\X}$ and denote by $F^\star$ the face of maximizers of the linear function $\mu\mapsto c^\top \mu$ over $P$. 
Denote the set of neighboring vertices of a vertex $\mu$ by $N(\mu)$ and let $\lVert \cdot\rVert\colon\R^{\X}\to\R_{\ge0}$ be an arbitrary semi-norm. 
Then either $F^\star = P$ or with $\frac{c}{0}\coloneqq+\infty$ for $c>0$, we have 
\begin{equation}
    \Delta\coloneqq \min \left\{ \frac{c^\top\mu^\star - c^\top\mu }{\lVert \mu^\star - \mu \rVert} : \mu^\star \in\operatorname{vert}(F^\star), \mu\in N(\mu^\star)\setminus F^\star \right\} > 0,
\end{equation}
and further 
\begin{equation}\label{eq:gapVSNorm-2}
     c^\top \mu^\star - c^\top \mu \ge \Delta \cdot \inf_{\mu^\star\in F^\star}  \lVert \mu^\star - \mu\rVert \quad \text{for all } \mu\in P.
\end{equation}
\end{restatable}
\begin{proof}
If $F^\star \ne P$, then $c^\top \mu^\star-c^\top \mu >0$ for some vertex $\mu$, which implies $\Delta>0$. 
To simplify notation we denote the set $E\coloneqq\{\mu-\mu^\star: \mu\in N(\mu^\star) \setminus F^\star, \mu^\star\in \operatorname{vert}(F^\star)\}$ of edges such that exactly one of the two endpoints is contained in $F^\star$. 
Then, the polytope $P$ is contained in 
    \[ F^\star+C = \left\{ \mu^\star + \sum_{e\in E} \alpha_e e : \mu^\star\in F^\star, \alpha_e\ge0 \text{ for all } e\in E \right\}, \]
see \Cref{app:lem:cone} and hence we can write $\mu\in P$ as $\mu = \mu^\star + \sum_{e} \alpha_e e$ for some $\mu^\star\in F^\star$. 
Using the triangle inequality we obtain 
    \[ \Delta \lVert \mu^\star - \mu \rVert \le \Delta \sum_{e\in E}\alpha_e \lVert e \rVert \le - \sum_{e\in E} \alpha_e c^\top e = c^\top \mu^\star - c^\top\mu. \]
\end{proof}

\begin{lemma}\label{lem:KLEstimateL1}
Consider a finite set $\X$ and a probability distribution $\mu\in \Delta_\X$. Let $c>1$ and set $\delta\coloneqq \frac{c-1}{c+1} 
    \cdot \min\{ \mu_x : \mu_x>0 \}>0$. 
Then for all $\nu\in\Delta_\X$ satisfying $\lVert \mu-\nu\rVert_\infty\le\delta$ it holds that 
\begin{equation}\label{eq:KLvsL1}
    \KL (\mu, \nu) \le c \cdot \lVert \mu - \nu \rVert_{\TV }. 
\end{equation}
\end{lemma}
\begin{proof}
We bound the individual summands in the KL-divergence
    \[ \KL (\mu, \nu) = \sum_{x\in\X} \mu_x \log\left( \frac{\mu_x}{\nu_x}\right) = \sum_{x\in X} \mu_x \log\left( \frac{\mu_x}{\nu_x}\right), \]
where $X\coloneqq\{x\in\X : \mu_x>0\}$. 
If $\mu_x,\nu_x>0$ then 
\begin{align}\label{eq:klIndividual}
    \begin{split}
        \mu_x \log \left(\frac{\mu_x}{\nu_x}\right) & = \mu_x \left( \log(\nu_x+(\mu_x-\nu_x)) - \log(\nu_x) \right) \\
        & \le \mu_x \left( \log(\nu_x) + \frac{\mu_x-\nu_x}{\nu_x} - \log(\nu_x) \right) 
         = (\mu_x - \nu_x)\cdot\frac{\mu_x}{\nu_x},
    \end{split}
\end{align}
where we used the convexity $\log(t+h) \le \log(t) + h/t$ for $t>0, t+h>0$. 
We set $\varepsilon\coloneqq\frac{c-1}{2}\in(0,1)$, such that 
\[ \delta = \frac{\varepsilon}{1+\varepsilon}  
    \cdot \min\big\{ \mu_x : \mu_x>0 \big\}. 
    \]
If $\lVert \mu-\nu\rVert_\infty\le\delta$ then 
    \[ \nu_x \ge \mu_x - \delta \ge \mu_x \left(1-\frac{\varepsilon}{1+\varepsilon}\right) = \frac{\mu_x}{1+\varepsilon} \]
as well as
    \[ \nu_x \le \mu_x + \delta  \le \mu_x \left(1 + \frac{\varepsilon}{1+\varepsilon}\right) \le \mu_x \left(1 + \frac{\varepsilon}{1-\varepsilon}\right) = \frac{\mu_x}{1-\varepsilon} \]
and therefore $1-\varepsilon\le \frac{\mu_x}{\nu_x}\le1+\varepsilon$. 
If $\mu_x\ge\nu_x$ then 
    \[ 
    (\mu_x-\nu_x)\cdot \frac{\mu_x}{\nu_x} 
    \le (1+\varepsilon)(\mu_x-\nu_x) 
    = \mu_x-\nu_x + \varepsilon \lvert \mu_x - \nu_x \rvert , 
    \]
and if $\mu_x<\nu_x$ then 
\begin{equation}\label{eq:almostThere}
    (\mu_x-\nu_x)\cdot \frac{\mu_x}{\nu_x} \le (1-\varepsilon)(\mu_x-\nu_x) = \mu_x-\nu_x + \varepsilon \lvert \mu_x - \nu_x \rvert. 
\end{equation}
Together with~\eqref{eq:klIndividual} summing over $x$ yields 
\begin{equation}\label{eq:there}
    \KL (\mu, \nu) \le \sum_{x\in X} (\mu_x-\nu_x) + \varepsilon \sum_{x\in X} \lvert \mu_x - \nu_x \rvert \le \sum_{x\in X} (\mu_x-\nu_x) + 2 \varepsilon \lVert \mu-\nu\rVert_{\TV } . 
\end{equation}
It remains to estimate the first part. 
Setting $X^c\coloneqq\X\setminus X$, we have 
\[ 
\sum_{x\in X} (\mu_x-\nu_x) = \sum_{x\in\X} (\mu_x-\nu_x) -  \sum_{x\in X^c} (\mu_x-\nu_x) = -  \sum_{x\in X^c} (\mu_x-\nu_x) = \sum_{x\in X^c} \lvert  \mu_x-\nu_x \rvert  
\]
since $\mu_x=0$ for $x\in X^c$. 
Now we can estimate
\begin{align}\label{eq:lastStep}
    2\sum_{x\in X} (\mu_x-\nu_x) = \sum_{x\in X} (\mu_x-\nu_x) + \sum_{x\in X^c} \lvert  \mu_x-\nu_x \rvert \le \lVert \mu-\nu\rVert_1 = 2\lVert \mu-\nu\rVert_{\TV }.
\end{align}
Combining~\eqref{eq:there} and~\eqref{eq:lastStep} yields
\[ \KL (\mu, \nu) \le (1+2\varepsilon)\lVert \mu - \nu \rVert_{\TV } = c\cdot\lVert \mu - \nu \rVert_{\TV }. \]
\end{proof}

\begin{corollary}[Local KL-TV estimate]\label{cor:TV-KL}
Consider a finite set $\X$ and a probability distribution $\mu\in \Delta_\X$. 
Then for all $\nu\in\Delta_\X$ satisfying 
\begin{equation}
    \lVert \mu - \nu \rVert_\infty < \min\{ \mu_x : \mu_x>0 \} 
\end{equation}
it holds that 
\begin{equation}\label{eq:KLvsL1-2}
    \KL (\mu, \nu) \le \frac{\min\{ \mu_x : \mu_x>0 \} +  \lVert \mu-\nu\rVert_\infty}{\min\{ \mu_x : \mu_x>0 \} - \lVert \mu-\nu\rVert_\infty} \cdot \lVert \mu - \nu \rVert_{\TV }. 
\end{equation}
\end{corollary}
\begin{proof}
    This is a direct consequence of \Cref{lem:KLEstimateL1}. 
    Indeed, for $\varepsilon>0$ small enough we have $\lVert \mu - \nu \rVert_\infty \le \frac{2-\varepsilon}{2+\varepsilon} \cdot \min\{ \mu_x : \mu_x>0 \}$ and thus by \Cref{lem:KLEstimateL1} with $c=1+\varepsilon$ we have $\KL (\mu, \nu) \le (1+\varepsilon) \lVert \mu - \nu \rVert_{\TV }$. 
    Note that $\varepsilon>0$ was arbitrary. 
\end{proof}

Now we prove our main result on the convergence of Fisher-Rao gradient flows. 
\begin{proof}[Proof of \Cref{thm:convergenceFRGF}]
    Setting $\delta\coloneqq \min\{ \mu_x^\star : \mu_x^\star>0 \}$ and using \Cref{lem:TV-LP-estimate-withoutUniqueness} with $\lVert\cdot\rVert_{\TV}$ and \Cref{cor:TV-KL} we have 
    \begin{equation*}
        c^\top \mu^\star - c^\top \mu_t \ge \Delta \lVert \mu^\star - \mu_t \rVert_{\TV } \ge \Delta\cdot\frac{\delta -  \lVert \mu^\star-\mu_t\rVert_\infty}{\delta + \lVert \mu^\star-\mu_t\rVert_\infty} \cdot \KL (\mu^\star, \mu_t),
    \end{equation*}
    if $\lVert \mu^\star - \mu_t \rVert_\infty < \delta$.
    By \Cref{prop:sublinearRate} we have 
    \begin{align*}
        \lVert \mu^\star - \mu_t \rVert_\infty \le 2 \lVert \mu^\star - \mu_t \rVert_{\TV } \le \frac{2c^\top(\mu^\star-\mu_t)}\Delta \le \frac{2\KL (\mu^\star,\mu_0)}{\Delta \cdot t}.
    \end{align*}
    Hence, for $t>t_0$ we have $\lVert \mu^\star - \mu_t \rVert_\infty < \delta$. 
    In this case, we can estimate 
    \begin{align*}
        \frac{\delta -  \lVert \mu^\star-\mu_t\rVert_\infty}{\delta + \lVert \mu^\star-\mu_t\rVert_\infty} & \le \frac{\delta -  2 \KL (\mu^\star, \mu_0)\Delta^{-1} t^{-1}} {\delta + 2 \KL (\mu^\star, \mu_0)\Delta^{-1} t^{-1}} 
        = \frac{t - t_0}{t + t_0} \eqqcolon \kappa_t. 
    \end{align*}
    Thus for $t>t_0$ we have $c^\top \mu^\star - c^\top \mu_t \ge \Delta \kappa_t \KL (\mu^\star, \mu_t)$, and \Cref{prop:ratesTrajectories} together with 
    \begin{equation*}
        \int_{t_0}^t \frac{s-t_0}{s+t_0} \D s = s - 2 t_0 \log(s+t_0)\Big|^{s=t}_{s=t_0} = (t-t_0) - 2 t_0 \log\left( \frac{t+t_0}{2t_0} \right)
    \end{equation*}
    yield the result. 
\end{proof}

\subsection{Estimating the regularization error}

Using the central path property we can deduce an estimate on the regularization error from the convergence results for the Fisher-Rao gradient flow. 
If the uniform distribution is contained in $P$, $\mu_{\operatorname{Unif}}\in P$, then the claim follows simply by setting $\mu_0\coloneqq \mu_{\operatorname{Unif}}$ as  
\begin{equation}
    \KL (\mu, \mu_{\operatorname{Unif}}) = -H(\mu) + \log \lvert \X \rvert. 
\end{equation}
If the uniform distribution is not contained in $P$, we can choose its information projection as an initial distribution $\mu_0$ to the same effect. 
Indeed, recall that for
\begin{equation}\label{eq:information-projection}
    \mu_0 = \argmin_{\mu\in P} \KL (\mu, \mu_{\operatorname{Unif}})  = \argmax_{\mu\in P} H(\mu) 
\end{equation}
we have by the Pythagorean theorem that 
\begin{equation}
    \KL (\mu, \mu_{\operatorname{Unif}}) = \KL (\mu, \mu_0) + \KL (\mu_0, \mu_{\operatorname{Unif}})
\end{equation}
for all $\mu\in P$, see~\cite[Theorem 2.8]{ay2017information}. 
Now we can estimate the regularization error. 

\begin{proof}[Proof of \Cref{cor:regularizationError}]
    By the central path property the Fisher-Rao gradient flow $(\mu_t)_{t\ge0}$ satisfies $\mu_t = \argmax\left\{ c^\top \mu - t^{-1} \KL (\mu, \mu_0) : \mu\in P \right\}$. 
    If we choose $\mu_0$ as the information projection according to~\eqref{eq:information-projection} the Pythagorean theorem yields 
    \begin{equation*}
        \KL (\mu, \mu_0) = \KL (\mu, \mu_{\operatorname{Unif}}) + H(\mu_0) - \log \lvert \X \rvert = H(\mu_0) - H(\mu). 
    \end{equation*}
    This shows that $\mu_t = \argmax\left\{ c^\top \mu + t^{-1} H(\mu) : \mu\in P \right\}$, i.e., that $\mu_t$ is the solution of the entropy regularized linear program~\eqref{eq:LP:reg}. 
    Now the claim follows from \Cref{thm:convergenceFRGF} and  $\KL (\mu^\star, \mu_0) = H(\mu_0) - H(\mu^\star) \le R_H$. 
\end{proof}

\subsection{Non-unique maximizers}
Both \Cref{thm:convergenceFRGF} and \Cref{cor:regularizationError} are formulated under the assumption that the linear program~\eqref{eq:LP-inside-simplex} admits a unique solution. 
This is satisfied for almost all costs $c\in\R^{\X}$, however, it can be generalized to all costs. 

To proceed like in the proof with a unique maximizer, we need to identify the limit of $\mu_t$ in $F^\star$. 
For linear objective functions the limit $\mu^\star$ is the information projection of $\mu_0$ to $F^\star$, see~\cite[Corollary 4.8]{alvarez2004hessian}. 
We include a proof here for the sake of completeness. 

\begin{corollary}[Implicit bias of Fisher-Rao GF]\label{cor:implicitBias}
    Consider Setting~\ref{setting:LP-inside-Simplex} and denote the face of maximizers of the linear program~\eqref{eq:LP-inside-simplex} by $F^\star$. 
    Then it holds that 
    \begin{equation}
        \lim_{t\to+\infty}\mu_t = \mu^\star = \argmin_{\mu\in F^\star} \KL (\mu, \mu_0). 
    \end{equation}
    In words, the Fisher-Rao gradient flow converges to the information projection of $\mu_0$ to $F^\star$, i.e., it selects
    the optimizer that has the minimum KL-divergence from $\mu_0$. 
\end{corollary}
\begin{proof}

By compactness of $P$, the sequence $(\mu_{t_n})_{n\in\N}$ has at least one accumulation point for any $t_n\to+\infty$. 
Hence, we can assume without loss of generality that $\mu_{t_n}\to\hat\mu$ and it remains to identify $\hat\mu$ as the information projection $\mu^\star\in F^\star$. 

Surely, we have $\hat\mu\in F^\star$ as $c^\top \hat\mu = \lim_{n\to\infty} c^\top \mu_{t_n} = \max_{\mu\in P} c^\top \mu$ by \Cref{prop:sublinearRate}. 
Further, by the central path property we have for any optimizer $\mu'\in F^\star$ that 
\begin{equation*}
    c^\top \mu_t - t^{-1} \KL (\mu_t, \mu_0) \ge 
    c^\top \hat\mu - t^{-1} \KL (\mu', \mu_0) 
\end{equation*}
and therefore 
\begin{equation*}
    \KL (\mu', \mu_0)  - \KL (\mu_t, \mu_0) \ge t c^\top (\mu'-\mu_t) \ge0.  
\end{equation*}
Hence, we have  
\begin{align*}
    \KL (\hat\mu, \mu_0) = \lim_{n\to\infty} \KL (\mu_{t_n}, \mu_0)  \le \KL (\mu', \mu_0)
\end{align*}
and can conclude by minimizing over $\mu'\in F^\star$. 
\end{proof}

\begin{theorem}\label{thm:convergence-without-uniqueness}
Consider Setting~\ref{setting:LP-inside-Simplex}, assume that the linear program is non-trivial, i.e., that $F^\star\ne P$, where $F^\star$ denotes the face of optimizers, and denote the information projection of $\mu_0$ to $F^\star$ by $\mu^\star\in F^\star$ and set 
\begin{align}\label{eq:DeltaWithoutUniqueness}
    \Delta \coloneqq \min \left\{ \frac{c^\top \mu^\star - 
    c^\top \mu}{\lVert \mu^\star - \mu \rVert_{\TV }} : \mu^\star \in\operatorname{vert}(F^\star), \mu\in N(\mu^\star)\setminus F^\star  \right\}>0. 
\end{align}
Then for any $\kappa\in(0,\Delta)$ there is $t_\kappa\in\R_{\ge0}$ such that for any $t\ge t_\kappa$ we have 
\begin{align}\label{eq:convergence-withoutUniqueness}
    \begin{split}
        \KL (\mu^\star, \mu_t) & \le 
    \KL (\mu^\star, \mu_0) e^{-\kappa(t-t_\kappa)}
    \end{split}
\end{align}
and 
\begin{align}
    c^\top \mu^\star - c^\top \mu_t & \le \Delta \KL (\mu^\star, \mu_0) e^{-\kappa(t-t_\kappa)}. 
\end{align}
\end{theorem}
\begin{proof}
    \Cref{cor:implicitBias} shows that $\mu_t \to \mu^\star$. 
    Let $\mu_t^\star \in F^\star$ denote the $\lVert\cdot\rVert_{\TV}$-projection of $\mu_t$ onto $F^\star$, i.e., be such that 
        \[
            \lVert \mu_t - \mu_t^\star \rVert_{\TV} = \min_{\mu'\in F^\star} \lVert \mu' - \mu_t \rVert_{\TV} \to 0 \quad \text{for } t\to+\infty 
        \]
    as $\mu_t\to\mu^\star\in F^\star$. 
    Now we have 
    \begin{equation*}
        \lVert \mu_t^\star - \mu^\star \rVert_{\TV} \le \lVert \mu_t^\star - \mu_t  \rVert_{\TV} + \lVert \mu_t - \mu^\star \rVert_{\TV} \to 0 \quad \text{for } t\to+\infty 
    \end{equation*}
    and hence $\mu_t^\star\to \mu^\star$. 
    Note that $\mu^\star\in\operatorname{int}(F^\star)$, i.e., has maximal support in $F^\star$ and hence $\mu_t^\star\ll\mu^\star$, see \Cref{app:lem:support}. 
    Together with $\mu^\star_t\to\mu^\star$ this yields 
    \begin{align*}
        \delta_t \coloneqq \min\{ \mu_t^\star(x) : \mu_t^\star(x)>0 \} \to \min\{ \mu^\star(x) : \mu^\star(x)>0 \} > 0. 
    \end{align*}
    Combining \Cref{cor:TV-KL} and \Cref{lem:TV-LP-estimate-withoutUniqueness} yields 
    \begin{align*}
        \KL(\mu^\star_t, \mu_t) 
        \le \frac{\delta_t + \lVert \mu_t^\star - \mu_t \rVert_{\TV}}{\delta_t - \lVert \mu_t^\star - \mu_t \rVert_{\TV}} \cdot \Delta^{-1} (c^\top \mu^\star - c^\top \mu_t), 
    \end{align*}
    where the right hand side converges to $\Delta^{-1} (c^\top \mu^\star - c^\top \mu)$ for $t\to+\infty$. 
    Hence, for $\kappa<\Delta$ and $t$ large enough, we have 
    \begin{align*}
        \kappa \KL(\mu^\star,\mu_t) \le \kappa \KL( \mu_t^\star, \mu_t) \le c^\top \mu^\star - c^\top \mu_t,
    \end{align*}
    where we used that $\mu^\star$ is the information projection of $\mu_t$ to $F^\star$ and $\mu_t^\star\in F^\star$, 
    therefore establishing~\eqref{eq:strongConvexity}. 
    Now we can conclude utilizing \Cref{prop:ratesTrajectories}. 
\end{proof}

A bound on the time $t_\kappa$ could be obtained through a refinement of \Cref{lem:TV-LP-estimate-withoutUniqueness} showing $c^\top \mu^\star - c^\top \mu \ge \Delta \cdot  \lVert \mu^\star - \mu\rVert_{\TV}$ for the information projection $\mu^\star$ of $\mu\in P$ to $F^\star$. 
Another approach to control $t_\kappa$ is to quantify the convergence of $\mu^\star_t\to\mu^\star$. 

\begin{remark}[Estimating the regularization error]\label{rem:regularizationError-withoutUniqueness}
Just like before, we can estimate the regularization error with the same argument as in  \Cref{cor:regularizationError}. 
In this case, the guarantee \eqref{eq:convergence-withoutUniqueness} holds with the entropic radius $R_H$ instead of $\KL(\mu^\star, \mu_0)$. 
\end{remark}

\section{Convergence of Natural Gradient Flows}
In practice, it is often not feasible to perform optimization in the space of measures, and therefore one often resorts to parametric models. 
Natural gradients were introduced by S.\ Amari~\cite{amari1998natural} and are designed to mimic the Fisher-Rao gradient flow by preconditioning the Euclidean gradient in parameter space with the Fisher information matrix. 
To study natural gradient methods, we work in the following setting. 

\begin{setting}\label{setting:NG}
We consider a finite set $\X$ and a polytope $P = \Delta_\X \cap  \mathcal L$ with $P \cap \R^\X_{>0}\ne\emptyset$, where $\mathcal L\subseteq\R^{\X}$ is an affine space.
Further, we consider a differentiable parametrization $\R^p\to\operatorname{int}(P); \theta\mapsto \mu_\theta$ and a (possibly nonlinear) differentiable objective function $f\colon\R_{>0}^\X\to\R$, and write $f(\theta) = f(\mu_\theta)$. 
\end{setting}

We work in continuous time and consider the following evolution of parameters. 

\begin{definition}[Natural gradient flow] 
Consider Setting~\ref{setting:NG}. We call 
\begin{equation}\label{eq:NGF}
    \partial_t \theta_t = F(\theta_t)^+\nabla f(\theta_t) 
\end{equation}
the \emph{natural gradient flow}, where $F(\theta)^+$ denotes the pseudo-inverse of the Fisher information matrix with entries 
\begin{equation}
    F(\theta)_{ij} =  \sum_{x\in\X} \frac{\partial_i \mu_\theta(x) \partial_j \mu_\theta(x)}{\mu(x)} = g^{\FR}_{\mu_\theta}(\partial_i \mu_\theta, \partial_j \mu_\theta). 
\end{equation}
\end{definition}

\subsection{Compatible function approximation} 
In this subsection, and more precisely in
\Cref{prop:actualCompatibleFA}, we describe the natural gradient direction as the minimizer of a linear least squares regression problem with features $\phi_\theta(x) = \nabla_\theta \log \mu_\theta(x)$.
This can be used to estimate the natural gradient from samples drawn from $\mu_\theta$. 

In the context of reinforcement learning similar techniques, albeit for a different notion of natural gradient, have been developed under the name \emph{compatible function approximation}~\cite{sutton1999policy, kakade2001natural, agarwal2021theory}. 

The measure $\mu_t = \mu_{\theta_t}$ does not necessarily evolve according to the Fisher-Rao gradient flow on the polytope $P$ 
\eqref{eq:FRGF} even if $\theta_t$ satisfies the natural gradient flow 
in the parameter space \eqref{eq:NGF}. 
In the next lemma we describe the discrepancy between $\partial_t \mu_t = \partial_t \theta_t^\top \nabla_\theta \mu_{\theta_t}$ and the Fisher-Rao gradient 
$\nabla^{\FR}_P f(\mu_t)$. 

\begin{lemma}\label{prop:compatiblaFA}
Consider Setting~\ref{setting:NG} and a parameter evolution $\partial_t \theta_t = v_t$ and write $\mu_t = \mu_{\theta_t}$. 
Then we have 
\begin{equation}
     \left\lVert \partial_t \mu_t - \nabla^{\FR}_P  f(\mu_t) \right \rVert_{g^{\FR}_{\mu_t}}^2 = L(v_t, \theta_t) - C(\theta_t),
\end{equation}
where 
\begin{equation}\label{eq:definitionCALoss}
    L(w, \theta) \coloneqq \mathbb E_{\mu_{\theta}}\left[\left(w^\top \nabla_\theta\log \mu_{\theta}(x) - \nabla f(\mu_\theta)(x)\right)^2 \right]
\end{equation}
is an $l^2$-regression error and $C(\theta_t) \coloneqq \inf_{\nu\in TP} \left\lVert \nabla^{\FR} f(\mu_t) - \nu \right\rVert_{g^{\FR}_{\mu_{\theta_t}}}^2$ a projection error. 
\end{lemma}
\begin{proof}
The Fisher-Rao gradient $\nabla^{\FR}_P  f(\mu_t)$ of $f\colon P\to\R$ is the Fisher-Rao projection of the Fisher-Rao gradient $\nabla^{\FR} f(\mu_t)$ of $f\colon\R_{>0}^{\X}\to\R$ onto $TP$.
Hence, by the Pythagorean theorem, we have
\begin{align*}
    \left\lVert \partial_t \mu_t - \nabla^{\FR}  f(\mu_t) \right \rVert_{g^{\FR}_{\mu_t}}^2 = \left\lVert \partial_t \mu_t - \nabla^{\FR}_P  f(\mu_t) \right \rVert_{g^{\FR}_{\mu_t}}^2 + \left\lVert \nabla^{\FR}_P  f(\mu_t) - \nabla^{\FR}  f(\mu_t) \right \rVert_{g^{\FR}_{\mu_t}}^2. 
\end{align*}

Since $\nabla^{\FR}_P  f(\mu_t)$ is the projection of $\nabla^{\FR} f(\mu_t)$ to $TP$, we obtain 
\begin{align*}
    \left\lVert \partial_t \mu_t - \nabla^{\FR}_P  f(\mu_t) \right \rVert_{g^{\FR}_{\mu_t}}^2 = \left\lVert \partial_t \mu_t - \nabla^{\FR} f(\mu_t)\right \rVert_{g^{\FR}_{\mu_t}}^2 - C(\theta_t). 
\end{align*}
Further, by the chain rule, we have $\partial_t \mu_t = \partial_{t}\theta_t^\top \nabla_\theta \mu_{\theta_t}(x) = v_t^\top \nabla_\theta \mu_{\theta_t}(x)$. 
Using $\nabla^{\FR} f(\mu) = \nabla f(\mu)\odot \mu$ we conclude 
\begin{align*}
    \left\lVert \partial_t \mu_t - \nabla^{\FR}  f(\mu_t)
    \right\rVert_{g^{\FR}_{\mu_\theta}}^2 
   & = \left\lVert v_t^\top \nabla_\theta \mu_\theta - \nabla f(\mu_\theta)\odot \mu_\theta 
    \right\rVert_{g^{\FR}_{\mu_\theta}}^2 
   \\ & = \mathbb E_{\mu_\theta}\left[\frac{\left(v_t^\top \nabla_\theta \mu_\theta(x) - \nabla f(\mu_\theta)(x)\mu_\theta(x)\right)^2}{\mu_\theta(x)^2} \right] 
   \\ &  = \mathbb E_{\mu_\theta}\left[\left(v_t^\top \nabla_\theta\log \mu_{\theta}(x) - \nabla f(\mu_\theta)(x)\right)^2 \right] 
   = L(v_t,\theta) . 
\end{align*}
\end{proof}

The distance between $\partial_t \mu_t$ and the Fisher-Rao gradient $\nabla^{\textup{FR}}_P f(\mu_t)$ is up to a remainder term given by the least squares loss $L(v_t, \theta_t)$, where $v_t = \partial_t \theta_t$. 
The natural gradient is designed such that $\partial_t \mu_t$ is close to $\nabla^{\textup{FR}}_P f(\mu_t)$~\cite{amari1998natural} and hence
we can minimize the least squares loss $L(v, \theta_t)$ with respect to $v$ in order to approximate the natural gradient $v_t\approx F(\theta_t)^+ \nabla f(\theta_t)$. 
An important benefit of this formulation is that it can be used to estimate the natural gradient from data distributed according to $\mu_{\theta_t}$. 
We make this relation between the minimization of $L$ and the natural gradient explicit. 

\begin{proposition}[Compatible function approximation]\label{prop:actualCompatibleFA} 
Consider \Cref{setting:NG}, let $F(\theta)$ denote the Fisher-information matrix, and let $L$ be defined as in~\eqref{eq:definitionCALoss}. 
Then $v\in\R^p$ is a natural gradient at $\theta\in\R^p$, i.e., satisfies $F(\theta)v= \nabla_\theta f(\theta)$, if and only if 
\begin{equation}
    v \in \argmin_{w\in\R^p} L(w, \theta). 
\end{equation}
\end{proposition}
\begin{proof}
The objective function $L(w, \theta)$ is given, up to a constant, by 
\begin{equation*}
    \left\lVert w^\top \nabla_\theta \mu_\theta\right\rVert_{g^{\FR}_{\mu_\theta}}^2 -2 g_{\mu_\theta}^{\FR}(w^\top \nabla_\theta \mu_\theta, \nabla^{\FR} f(\mu_\theta)) = w^\top F(\theta) w - 2 \nabla f(\theta)^\top w. 
\end{equation*}
The global minimizes are characterized by the normal equation $F(\theta)w = \nabla f(\theta)$. 
\end{proof}

The term
\begin{equation}
    \varepsilon_t^2 = \min_{w\in \R^p} L(w, \theta_t) = \min_{w\in \R^p} \mathbb E_{\mu_{t}}\left[\left(w^\top \nabla_\theta\log \mu_{\theta_t}(x) - \nabla f(\mu)(x)\right)^2\right]
\end{equation}
is can be interpreted as an \emph{approximation error}. 
Note, however, that the precise nature of the least square loss $L$ is different from the one well-known in reinforcement learning as we discuss in more detail in~\Cref{rem:comparison-kakade}. 
Examining the objective $L(w,\theta)$ and using \Cref{prop:compatiblaFA} we see that the natural gradient flow minimizes the discrepancy between $\partial_t \mu_t$ and the Fisher-Rao gradient $\nabla^{\FR}_P f(\mu_t)$. In this case, the evolution $\partial_t \mu_t$ is given by the orthogonal projection of the Fisher-Rao gradient onto the tangent space of the parametrized model. 
A similar property holds for any natural gradient defined using a Riemannian metric on the polytope~\cite{amari2016information, van2023invariance, muller2023achieving}. 

\begin{corollary}[Projection property]\label{prop:projectionProperty}
    Consider a solution $(\theta_t)_{t\in[0, T)}$ of the natural gradient flow~\eqref{eq:NGF}.
    We denote the projection with respect to the Fisher-Rao metric onto the generalized tangent space  
    \[ 
    T_\theta P \coloneqq \operatorname{span}\{ \partial_{\theta_i} \mu_\theta : i=1, \dots, p \} = \{ w^\top \nabla_\theta \mu_\theta : w\in\R^p \} \subseteq TP 
    \]
    by $P^{\FR}_\theta$. 
    Then it holds that 
    \begin{equation}
        \partial_t \mu_t = P^{\FR}_{\theta_t}(\nabla^{\FR}_P f(\mu_t) ). 
    \end{equation}
    In particular, if $T_{\theta_t} P = TP$ then $\partial_t \mu_t = \nabla^{\FR}_P f(\mu_t)$. 
\end{corollary}
\begin{proof}
By \Cref{prop:actualCompatibleFA} the natural gradient direction $v_t$ is a minimizer of $L(\cdot, \theta_t)$. 
By \Cref{prop:compatiblaFA} this yields 
\begin{align*}
    \lVert \partial_t \mu_t - \nabla^{\FR}_P f(\mu_t) \rVert_{g^{\FR}_{\mu_t}} & = 
    \min_{w\in \R^p} \left\lVert w^\top \nabla_\theta d_{\theta} - \nabla^{\FR}_P f(\mu_t) \right\rVert_{g^{\FR}_{\mu_t}} 
    \min_{\nu\in T_\theta P} \left\lVert \nu - \nabla^{\FR}_P f(\mu_t) \right\rVert_{g^{\FR}_{\mu_t}}.
\end{align*}
In particular, this shows that $\partial_t\mu_t$ is the projection of $\nabla^{\FR}_P f(\mu_t)$ onto $T_\theta P$. 
\end{proof}

\subsection{Convergence of natural gradient flows} 

We start with a generalization of \Cref{prop:sublinearRate} to cover cases where the evolution of $\mu_t$ only approximately follows the Fisher-Rao gradient flow. 

\begin{proposition}[A perturbed convergence result]\label{prop:perturbedConvergence}
    Consider Setting~\ref{setting:LP-inside-Simplex}, a differentiable curve $\mu\colon[0,\infty)\to\operatorname{int}(P)$ and a differentiable convex objective $f\colon\R_{>0}^\X$. 
    Assume that $f$ admits a maximizer $\mu^\star$ over $P$ with value $f^\star$. 
    It holds that \begin{equation}\label{eq:PerturbedSublinearConvergence}
        f^\star - f(\mu_t) \le \frac{\KL (\mu^\star, \mu_0) - \KL (\mu^\star, \mu_t)}{t} + t^{-1} \int_{0}^t \varepsilon_s \delta_s \D s,
    \end{equation}
    where $\delta_t^2 \coloneqq \chi^2(\mu^\star, \mu_t)$ and $\varepsilon_t^2 \coloneqq \left\lVert \nabla^{\FR}_P f(\mu_t)-\partial_t \mu_t\right\rVert_{g^{\FR}_{d_t}}^2$. 
\end{proposition}
\begin{proof}
    We compute 
    \begin{align*}
        \partial_t \KL (\mu^\star, \mu_t) & = - \partial_t \phi(\mu_t) - \partial_t \langle \nabla \phi(\mu_t), \mu^\star - \mu_t \rangle 
        = \langle \nabla^2 \phi(\mu_t)\partial_t \mu_t, \mu_t - \mu^\star \rangle 
        \\ & 
        = g_{\mu_t}^{\FR}(\partial_t \mu_t, \mu_t - \mu^\star) 
        \\ & = g_{\mu_t}^{\FR}(\nabla^{\FR}_P f(\mu_t), \mu_t - \mu^\star) + g_{\mu_t}^{\FR}(\nabla^{\FR}_P f(\mu_t)-\partial_t \mu_t, \mu_t - \mu^\star) 
        \\ & = \nabla f(\mu_t)^\top( \mu_t - \mu^\star) + g_{\mu_t}^{\FR}(\nabla^{\FR}_P f(\mu_t)-\partial_t \mu_t, \mu_t - \mu^\star) 
        \\ & \le \nabla f(\mu_t)^\top( \mu_t - \mu^\star) + \varepsilon_t\delta_t 
    \le f(\mu_t) - f(\mu^\star) + \varepsilon_t\delta_t, 
    \end{align*}
    where we used \Cref{prop:compatiblaFA} and \Cref{prop:actualCompatibleFA} as well as $\lVert \mu_t - \mu^\star \rVert_{g_{\mu_t}^{\FR}}^2 = \chi^2(\mu^\star, \mu_t)$. 
    Integration and rearranging now yields~\eqref{eq:PerturbedSublinearConvergence}. 
\end{proof}

If $(\mu_t)_{t\ge0}$ solves the Fisher-Rao gradient flow, we have $\varepsilon_t=0$ and recover  \Cref{prop:sublinearRate}. 
For natural gradient flows, we obtain the following result. 

\begin{corollary}\label{cor:convergenceGeneral}
Consider Setting~\ref{setting:NG} and a solution $(\theta_t)_{t\ge0}$ of the natural gradient flow~\eqref{eq:NGF} for a convex objective $f$ and set $\mu_t\coloneqq\mu_{\theta_t}$. 
Then~\eqref{eq:PerturbedSublinearConvergence} holds with 
\begin{equation}
    \varepsilon_t^2 \le \min_{w\in\R^p} \mathbb E_{\mu_{\theta}}\left[\left(w^\top \nabla_\theta\log \mu_{\theta_t}(x) - \nabla f(\mu_t)(x)\right)^2 \right]. 
\end{equation}
\end{corollary}
\begin{proof}
    Combine \Cref{prop:perturbedConvergence} with \Cref{prop:compatiblaFA} and \Cref{prop:actualCompatibleFA}. 
\end{proof}

\begin{remark}[Baseline] 
In reinforcement learning, baselines are often used when estimating the natural policy gradient from samples to reduce the variance of the estimates~\cite{10.5555/2074022.2074088}. 
This amounts to projecting the gradient of the objective to the tangent space of the model. 
In our setting, this corresponds to projecting $\nabla f(\mu)\odot\mu$ to the tangent space $TP$ with respect to the Fisher-Rao metric $g^{\FR}_\mu$, see also~\cite[Subsection 4.1.1]{suarez2023perspectives}. 
In the special case $P=\Delta_\X$ the Fisher-Rao projection of $\nabla f(\mu)\odot\mu$ is given by $\nabla f(\mu)\odot\mu - \kappa\mu$, where $\kappa = \sum_{x}\nabla f(\mu)(x)$ and the corresponding compatible function approximation objective is given by 
\begin{align*}
    \tilde L(w, \theta) \coloneqq \mathbb E_{\mu_{\theta}}\left[\left(w^\top \nabla_\theta\log \mu_{\theta}(x) - (\nabla f(\mu)(x)-\kappa)\right)^2 \right]. 
\end{align*}
\end{remark}

\subsection{Global convergence for multi-player games} 
With function approximation \Cref{cor:convergenceGeneral} ensures sublinear convergence $O(\frac1t)$ up to a remainder compared to the linear rate global convergence guarantee of the Fisher-Rao gradient flow. 
With general function approximation, it is however not possible to guarantee global convergence~\cite{bhandari2024global} and also for other natural gradient methods the linear convergence guarantees are lost when working with function approximation~\cite{alfano2022linear, cayci2024convergence} unless one uses regularization. 
Here, we identify a scenario, where despite being in a function approximation setting, we can ensure global linear convergence. 

For a rich enough parametrization \Cref{prop:projectionProperty} ensures that $(\mu_{\theta_t})_{t\ge0}$ follows the Fisher-Rao gradient flow in which case \Cref{thm:convergenceFRGF} implies the linear convergence of the natural gradient flow~\eqref{eq:NGF}. 
A common example is the softmax parametrization $\mu_\theta(x)\propto e^{\theta(x)}$. 
For multi-player games with suitable payoff structure, the dynamics of the individual players decouple~\cite{boll2024geometric}, which allows us to show global convergence for models with exponentially fewer parameters than the softmax parametrization. 

\begin{theorem}\label{thm:multiPlayer}
    Consider a differentiable parametrization of conditional probabilities $\{m_\theta : \theta\in\mathbb R^p\}=\operatorname{int}(\Delta_\X^n)$, where $n\in\N$ and $\X$ is a finite set, and suppose that $\operatorname{span}\left\{ \partial_{\theta_i} m_\theta : i=1, \dots, p \right\} = T \Delta_\X^n$ for all $\theta\in\R^p$.
    Define a corresponding parametric independence model as 
    \begin{align}
        {\mu}_{{\theta}}({x}) \coloneqq \prod_{i=1}^n m_\theta(x_i|i) \quad \text{for all } {x} \in\X^n. 
    \end{align}
    Further, consider 
    \begin{align}\label{eq:factorization}
    c\in\operatorname{span} \Big\{ \mathds{1}_{\X} \otimes\dots\otimes \underset{i\text{-th}}{\delta_x} \otimes\dots\otimes \mathds{1}_\X : x\in\X, i=1, \dots, n \Big\} \subseteq \R^{\X^n} 
    \end{align}
    and the linear payoff $f({\mu})={c}^\top{\mu}$ and the natural gradient flow~\eqref{eq:NGF}. 
    Then $(\mu_{\theta_t})_{t\ge0}$ solves the Fisher-Rao gradient flow in $\Delta_{\X^n}$ 
    and hence, we have 
    \begin{align}
        \mu_t(x) = \frac{e^{tc(x)}}{\sum_{x'} e^{tc(x')}} \quad \text{for all } x\in\X^n. 
    \end{align}
\end{theorem}
\begin{proof}
The Segre embedding $\Delta_\X^n\to\Delta_{\X^n}$, $(\mu_i)_{i=1, \dots, n} \mapsto \otimes_{i=1}^n \mu_i$ is an isometry with respect to the product Fisher-Rao metric, i.e., the sum of the Fisher metrics over the individual factors, and the Fisher-Rao metric~\cite{montufar2014fisher, boll2024geometric}. 
In particular, this implies that $(\mu_{\theta_t})_{t\ge0}$ solves the Fisher-Rao gradient flow with respect to $f$ restricted the independence model 
\begin{align*}
    \mathcal{I} \coloneqq \left\{ \bigotimes_{i=1}^n \mu_i : \mu_i\in\Delta_\X \text{ for } i=1, \dots, n \right\} \subseteq \Delta_{\X^n}, 
\end{align*}
as $\partial_t \mu_t = P_{T_{\mu_t}\mathcal I} \nabla^{\FR} f(\mu_t) = \nabla^{\FR} f|_{\mathcal I}(\mu_t)$, see \cite{van2023invariance}. Condition~\eqref{eq:factorization} implies that $f$ factorizes along the marginalization map and hence the independence model $\mathcal I$ is invariant under the Fisher-Rao gradient flow~\cite{boll2024geometric}. 
Thus, $(\mu_{\theta_t})_{t\ge0}$ solves the Fisher-Rao gradient flow with potential $f$ in $\Delta_{\X^n}$, which can be solved explicitly~\cite{weed2018explicit}. 
\end{proof}

Note that a model parametrizing $\Delta_\X^n$ only requires $n (\lvert \X\rvert-1)$ parameters, whereas a model parametrizing the joint distributions $\Delta_{\X^n}$ requires $\lvert \X \rvert^n-1$ parameters. 
However, we require the cost vector $c$ to lie in an $n\lvert X \rvert$-dimensional subspace of $\R^{\X^n}$. 

\section{Convergence of State-Action Natural Policy Gradients}\label{sec:state-action-NPG}

Having studied general linear programs we now turn to the reward optimization problem in infinite-horizon discounted Markov decision processes. 
Reward optimization is well known to be equivalent to a linear program and the state-action natural policy gradient flow corresponds to the Fisher-Rao gradient flow inside the state-action polytope~\cite{kallenberg1994survey, muller2023geometry}. 
We give a short overview of the required notions and refer to~\cite{hernandez2012discrete} for a thorough introduction to Markov decision processes.  

In Markov decision processes (MDPs), we are concerned with controlling the state $s\in\bS$ of some system through an action $a\in\A$ in order to achieve an optimal behavior over time. 
The evolution of the system is described by a Markov kernel $P\in \Delta_{\bS}^{\bS \times \A}$, where $P(s'|s,a)$ denotes the probability of transitioning from state $s$ to $s'$ under action $a$. 
Here, we work with finite state and action spaces $\bS$ and $\A$. 
A \emph{(stochastic) policy} is a Markov kernel $\pi\in\Delta_\A^\bS$, where $\pi(a|s)$ denotes the probability of selecting action $a$ when in state $s$. 
For a fixed policy $\pi\in\Delta_\A^\bS$ and an initial distribution $\mu\in\Delta_\bS$ we obtain a Markov process over $\bS\times \A$ according to $S_0 \sim \mu$ and 
\begin{equation}
    A_t \sim \pi(\cdot|S_t), \quad S_{t+1} \sim P(\cdot|S_t, A_t) \quad \text{for } t\in\mathbb N , 
\end{equation}
and we denote its law by $\mathbb P^{\pi, \mu}$. 
We consider a \emph{instantaneous reward vector} $r\in\R^{\bS\times\A}$ indicating how favorable a certain state and action combination is. 
As a criterion for the performance of a policy $\pi$ we consider the \emph{infinite horizon discounted reward} 
\begin{equation}
    R(\pi) \coloneqq (1-\gamma) \mathbb E_{\mathbb P^{\pi, \mu}}\left[ \sum_{t\in\mathbb N} \gamma^t r(S_t, A_t) \right],
\end{equation}
where the \emph{discount factor} $\gamma\in[0, 1)$ is fixed and ensures convergence. 
The reward optimization problem is given by 
\begin{equation}
    \max R(\pi) \quad \text{subject to } \pi\in\Delta_\A^\bS. 
\end{equation}
An important role in Markov decision processes play the \emph{state-action distributions} $d^\pi\in\Delta_{\bS\times\A}$, which are given by 
\begin{equation}
    d^\pi(s,a) \coloneqq (1-\gamma)  \sum_{t\in\mathbb N} \gamma^t \mathbb P^{\pi, \mu}(S_t=s, A_t=a). 
\end{equation}
They determine the reward as $R(\pi) = \sum_{s\in\bS, a\in\A} r(s,a) d^\pi(s, a) = r^\top d^\pi$. 
The set of state-action distributions has been characterized as a polytope, see~\cite{derman1970finite}. 

\begin{proposition}[State-action polytope]
    The set $\cD = \{d^\pi : \pi\in\Delta_\A^\bS\} \subseteq\Delta_{\bS\times\A}$ of state-action distributions is a polytope given by 
    \begin{equation}
        \cD = \Delta_{\bS\times\A} \cap \left\{ d\in \R^{\bS\times\A} : \ell_s(d) = 0 \text{ for all } s\in\bS \right\},
    \end{equation}
    where the defining linear equations are given by 
    \begin{equation}
        \ell_s(d) = \sum_{a\in\A} d(s,a) - \gamma \sum_{s'\in\bS, a'\in\A} P(s|s', a') d(s', a') - (1-\gamma) \mu(s).
    \end{equation}
\end{proposition}

We refer to $\cD$ as the \emph{state-action polytope}. 
This leads to the linear programming formulation of Markov decision processes~\cite{kallenberg1994survey}, given by\footnote{Sometimes, this is referred to as the dual linear programming formulation of Markov decision processes, where the primal linear program has the optimal value function as its solution.} 
\begin{equation}\label{eq:LP-MDP}
    \max r^\top d \quad \text{subject to } d\in \cD.  
\end{equation}
The state-action polytope $\cD = \Delta_{\bS\times\A}\cap\mathcal L$ falls under the class of polytopes studied in \Cref{sec:linearConvergeFRGF}. 
Given a state-action distribution $d\in\cD$, we can compute a corresponding policy $\pi\in\Delta_{\A}^\bS$ with $d = d^\pi$ by conditioning,  
\begin{equation}
    \pi(a|s) = \frac{d(s,a)}{\sum_{a'\in\A} d(s,a')} \quad \text{for all } a\in\A, s\in\bS, 
\end{equation}
if this is well-defined, see~\cite{muller2022pomdps, laroche2023occupancy}, which leads us to the following assumption.

\begin{asu}[State exploration]\label{s_asu:exploration} 
For any policy $\pi\in\Delta_\A^\bS$ the discounted state distribution is positive, i.e., $\sum_{a\in\A}d^\pi(s,a)>0$ for all $s\in\bS$. 
\end{asu}

This assumption is satisfied if $\mu(s)>0$ for all $s\in\bS$ as $\sum_{a\in\A}d^\pi(s,a)\ge (1-\gamma)\mu(s)$. 
This assumption is standard in linear programming approaches to Markov decision processes; policy gradient methods can fail to converge if it is violated~\cite{kallenberg1994survey, mei2020global}. 

Policy optimization algorithms parameterize the policy $\pi_\theta$ and optimize  $\theta$. 
As we study gradient-based approaches we work under the following assumption. 

\begin{asu}[Differentiable parametrization]\label{s_asu:differentiable}
We consider a differentiable policy parametrization  $\R^p\to\operatorname{int}(\Delta_\A^\bS); \theta\mapsto \pi_\theta$. 
\end{asu}
We consider continuous-time natural policy gradient methods that optimize the parameters $\theta$ of a parametric policy $\pi_\theta$ 
according to 
\begin{equation}\label{eq:NPGF}
    \partial_t \theta_t = G(\theta_t)^+ \nabla R(\theta_t),
\end{equation}
where we write $R(\theta) = R(\pi_\theta)$. 
Here $G(\theta)$ denotes a Gramian matrix with entries $G(\theta)_{ij} = g_{d_\theta}(\partial_{\theta_i}d_\theta, \partial_{\theta_j}d_\theta)$, where we write $d_\theta = d^{\pi_\theta}$ and $g_d$ denotes a Riemannian metric on the state-action polytope $\cD$. 
In this context, the matrix $G(\theta)$ is referred to as a preconditioner. 
Various choices have been proposed for $G(\theta)$, for example Kakade~\cite{kakade2001natural} suggested 
\begin{equation}
    G_{\operatorname{K}}(\theta) \coloneqq \sum_{s} d_\theta(s)  \sum_{a} \frac{\partial_{\theta_i} \pi_\theta(a|s) \partial_{\theta_j} \pi_\theta(a|s)}{\pi_\theta(a|s)}, 
\end{equation}
which is a weighted sum of Fisher-information matrices over the individual states~\cite{kakade2001natural, bagnell2003covariant, peters2008natural}.
This has been studied extensively in the literature, see for example~\cite{agarwal2021theory, cen2021fast, cayci2024convergence, khodadadian2022linear}, and we refer to it as the \emph{Kakade NPG}. 
We focus on the so-called \emph{state-action natural policy gradient} given by the Fisher information matrix of the state-action distribution~\cite{morimura2008new},  
\begin{equation}
    G_{\operatorname{M}}(\theta)_{ij} \coloneqq F(\theta)_{ij} = \sum_{s, a} \frac{\partial_{\theta_i} d_\theta(s,a) \partial_{\theta_j} d_\theta(s,a)}{d_\theta(s,a)} = g_{d_\theta}^{\FR}(\partial_{\theta_i} d_\theta, \partial_{\theta_j} d_\theta) . 
\end{equation}
This choice was observed to reduce the severity of plateaus, was used to design a natural actor-critic method~\cite{morimura2009generalized}, and is closely connected to the trust region method known as relative entropy policy search (REPS)~\cite{peters2010relative}. 

\subsection{Convergence guarantees} 
Now that we have built a convergence theory for general natural gradient flows we elaborate on the consequences for state-action natural policy gradients. 

\begin{corollary}[Sublinear convergence under function approximation]\label{cor:sublinearConvergenceNPG}
Consider a finite discounted Markov decision process, suppose \Cref{s_asu:exploration} and \Cref{s_asu:differentiable} hold, and consider a solution of the natural policy gradient flow~\eqref{eq:NPGF} for $G=G_{\operatorname{M}}$ and set $R^\star \coloneqq \max_{\pi\in\Delta_{\A}^\bS} R(\pi)$. 
Then it holds that 
    \begin{equation}\label{eq:sublinearConvergenceNPGFA}
        R^\star - R(\theta_t) \le \frac{\KL (d^\star, d_0)}{t} + t^{-1} \int_{0}^t \delta_s \varepsilon_s \D s,
    \end{equation}
    where $\delta_t^2 \coloneqq \chi^2(d^\star, d_t)$ and $\varepsilon_t^2\coloneqq \min_{w\in \R^p} \left\lVert \nabla^{\FR}_{\cD}f(\mu_t)-w^\top \nabla_\theta d_{\theta_t} \right\rVert_{g^{\FR}_{d_t}}^2$ and  
    \begin{equation}
        \varepsilon_t^2 \le \min_{w\in \R^p} \mathbb E_{d_{t}}\left[\left(w^\top \nabla_\theta\log d_{\theta_t}(s,a) - r(s,a)\right)^2\right].
    \end{equation}
\end{corollary}
\begin{proof}
    This is \Cref{cor:convergenceGeneral} for state-action natural policy gradients. 
\end{proof}

\begin{remark}[Inexact gradient evaluations]
    If the parameters follow the evolution $\partial_t \theta_t = v_t$, then we can apply \Cref{prop:perturbedConvergence} to see that~\eqref{eq:sublinearConvergenceNPGFA} remains valid with 
    \begin{equation*}
        \varepsilon_t^2\coloneqq\left\lVert \nabla^{\FR}_\cD f(\mu_t)-v_t^\top \nabla_\theta d_{\theta_t} \right\rVert_{g^{\FR}_{d_t}}^2 \le  \mathbb E_{d_{t}}\left[\left(v_t^\top \nabla_\theta\log d_{\theta_t}(s,a) - r(s,a)\right)^2\right]. 
    \end{equation*}
\end{remark}

\begin{remark}[Comparison to Kakade's natural policy gradient]\label{rem:comparison-kakade}
For Kakade's natural policy gradient in discrete time without entropy regularization in the function approximation regime, 
the value converges as 
$O(\frac1t)$ up to a remainder stemming from function approximation \cite{alfano2023novel}. 
Compared to~\eqref{eq:sublinearConvergenceNPGFA}, the $O(\frac1t)$ involves a conditional KL term corresponding to the Kakade geometry, which is induced by the conditional entropy rather than the entropy. 
More importantly, however, it comes with a multiplicative distribution mismatch coefficient, where it is unclear whether it remains bounded during optimization. 
However, it is unclear whether this is inherent to Kakade's natural policy gradient or an artifact of the proof. 
The remainder term in~\cite{alfano2023novel} again depends on the distribution mismatch and on a concentrability coefficient similar to $\chi^2(d^\star, d_t)$.  
Another difference between Kakade's and state-action natural policy gradients is that the compatible function approximation regresses the (estimated) $Q$ or advantage function instead of the reward vector $r$, therefore leading to a different approximation error $\tilde\varepsilon_t$. 
Further, Kakade's natural policy gradient without entropy regularization in a function approximation setting has been shown to converge linearly when using geometrically increasing step sizes~\cite{xiao2022PGjmlr, alfano2022linear, yuan2022linear}. 

Finally, entropy regularization with strength $\lambda$ leads to $O(e^{-\lambda t})$ convergence up to a remainder term, where the same $\chi^2$-divergence appears in the remainder term albeit with a different approximation error term~\cite{cayci2024convergence}. 
\end{remark}

We have studied general policy parameterizations and have seen that the corresponding state-action distributions evolve according to the projection of the Fisher-Rao gradient flow. 
A particularly nice case is given by parameterizations that are rich enough to express all policies as in this case the state-action distributions exactly evolve according to the Fisher-Rao gradient flow. 
This is why we consider the following condition for policy parameterizations. 

\begin{definition}[Regular tabular parametrization]
We say that a differentiable parametrization  $\R^p\to\operatorname{int}(\Delta_\A^\bS); \theta\mapsto \pi_\theta$ is a \emph{regular tabular parametrization} if it is surjective and satisfies 
\begin{equation}
    \operatorname{span}\{\partial_{\theta_i}\pi_\theta : i=1, \dots, p\} = T \Delta_\A^\bS \quad \text{for all } \theta\in\R^p. 
\end{equation}
\end{definition}

Since $\pi\mapsto d^\pi$ is a diffeomorphism between $\operatorname{int}(\Delta_\A^\bS)$ and $\operatorname{int}(\cD)$, see~\cite{muller2023geometry}, we have 
\begin{equation*}
    \operatorname{span}\{\partial_{\theta_i} d_\theta : i=1, \dots, p\} = T \cD\quad \text{for all } \theta\in\R^p
\end{equation*}
for a regular policy parametrization. 

Regular parametrizations include the following common examples: 
\begin{itemize}
    \item \emph{Expressive exponential families:} 
    For a feature map $\phi\colon\bS\times\A\to\R^p$ and $\theta\in\R^p$ we consider the log-linear policy $\pi_\theta(a|s) \propto e^{\theta^\top\phi(s,a)}$.
    This provides a regular tabular parametrization if $\operatorname{rank}\{ \phi(s,a) : s\in\bS, a\in\A \} = \lvert \bS \rvert \cdot \lvert \A \rvert$, see~\cite[Remark 2.4]{rauh2011thesis}. 
    In particular, this includes 
    \emph{tabular softmax} policies, where $\pi_\theta(a|s)\propto e^{\theta_{s,a}}$ which is the arguably most commonly studied policy class. 
    \item \emph{Escort transform:}
    The so-called \emph{escort transform} $\pi_\theta(a|s) \propto \lvert\theta_{s,a}\rvert^p$, for a parameter $p\ge1$ 
    was introduced in~\cite{mei2020escaping} to reduce the plateaus of vanilla policy gradients when working with softmax policies. 
\end{itemize}

For regular tabular parameterizations, the state-action distributions $d_t$ evolve according to the Fisher-Rao gradient flow inside the state-action polytope $\cD$. 
Hence, we can apply our general convergence theory to obtain the following result. 

\begin{corollary}[Linear convergence for tabular parametrizations]\label{thm:linearConvergenceTabular}
Consider a finite discounted Markov decision process, suppose \Cref{s_asu:exploration} and \Cref{s_asu:differentiable} hold, and consider a solution of the natural gradient flow~\eqref{eq:NPGF} for a regular tabular parametrization and write $\mu_t = d_{\theta_t}$. 
Then $(\mu_t)_{t\ge0}$ solves the Fisher-Rao gradient flow of the linear program~\eqref{eq:LP-MDP} and hence \Cref{thm:convergenceFRGF} and \Cref{thm:convergence-without-uniqueness} hold.
This implies $O(e^{-\Delta t + \kappa\log t})$ convergence for some $\kappa\ge0$, where
\begin{align}\label{eq:defDelta}
    \begin{split}
        \Delta = \min \left\{ \frac{R^\star - R(\pi)}{\lVert d^{\pi^\star} - d^\pi \rVert_{\TV }} : \begin{array}{c}
            \pi \text{ is deterministic and agrees with } \\ \text{a deterministic optimal policy } \pi^\star \\ \text{ on all but one state}  
        \end{array}\right\} > 0 
    \end{split}
\end{align}
and $R^\star = \max_{\pi\in\Delta_{\A}^\bS} R(\pi)$ denotes the optimal reward. 
\end{corollary}
\begin{proof}
The neighbors in $\cD$ and $\Delta_\A^\bS$ correspond to each other~\cite{muller2022pomdps}. 
Hence, $d^\pi$ is a neighbor of $d^\star$ if $\pi$ is deterministic and agrees with $\pi^\star$ on all but one state. 
\end{proof}

\begin{remark}[Comparison to Kakade's NPG]
    Much like the state-action natural policy gradient, Kakade's natural policy gradient with exact gradient evaluations has been shown to converge linearly without the need for entropy regularized setting~\cite{khodadadian2022linear}. 
    Here, the discrete-time setting is studied and NPG is interpreted as soft policy iteration. 
    This is used to show a convergence rate of $R^\star - R(\pi_k) = O(e^{-c k})$ for any $c \in(0, \Delta_{\operatorname{K}})$, where $\Delta_{\operatorname{K}} \coloneqq - (1-\gamma)^{-1}\max\left\{ A^\star(s,a) : a\ne a_s^\star \right\} \ge \Delta$, where $a^\star_s$ denotes the optimal action in state $s$. Indeed, by the performance difference lemma, we have 
\begin{align*}
    \Delta & = \min_{d\in N(d^\star)}  -\frac{d^\top A^\star}{(1-\gamma)\cdot \lVert d^\star - d \rVert_{\TV }}   = - (1-\gamma)^{-1}\max_{d\in N(d^\star)}  \frac{d^\top A^\star}{\lVert d^\star - d \rVert_{\TV }} . 
\end{align*}
Note that $d\in N(d^\star)$ can be associated with a policy $\pi$ that agrees with $\pi^\star$ on all but one state, and we write $\pi(a_0|s_0)=1$ for $a_0\ne a_{s_0}^\star$ for some $s_0$ and $\pi(a_s^\star|s)=1$ for $s\ne s_0$. 
Since $A^\star(s, a_s^\star) = 0$ we have $d^TA^\star = d(s_0) A^\star(s_0, a_0)\le0$ and estimate 
\begin{align*}
    2\lVert d^\star - d \rVert_{\TV } & = \sum_{s\ne0} \lvert d^\star(s) - d(s) \rvert + d^\star(s_0) + d(s_0) 
    \\ &
     \ge \sum_{s\ne0} (d^\star(s) - d(s)) + d^\star(s_0) + d(s_0) 
    \\ & = (1-d^\star(s_0)) - (1-d(s_0)) + d^\star(s_0) + d(s_0)  = 2d(s_0). 
\end{align*}
Overall, this yields $\frac{d^\top A^\star}{\lVert d^\star - d \rVert_{\TV }} \ge {A^\star(s_0, a_0)}$ and therefore 
\begin{align*}
\Delta \le - (1-\gamma)^{-1}\max\left\{ A^\star(s,a) : a\ne a_s^\star \right\} = \Delta_{\operatorname{K}}.
\end{align*}
Hence, the guaranteed convergence rate of Kakade's NPG is faster compared to the rate we provide for the state-action natural policy gradient. 
An essentially matching lower bound has been established for Kakade's NPG in~\cite{muller2024essentially}, whereas a matching lower bound is missing for state-action natural policy gradients. 
In our computational example, both converge at the same exponential rate $O(e^{-\Delta_{\operatorname{K}}t})$ even if $\Delta < \Delta_{\operatorname{K}}$. 
\end{remark}

\begin{remark}[Implicit bias]\label{rem:implicit-bias}
In particular, \Cref{thm:linearConvergenceTabular} guarantees that in the case of multiple optimal policies, the gradient flows corresponding to state-action natural policy gradients converge exponentially fast towards the information projection $d^\star$ of $d^{\pi_0}$ to the set of maximizes $D^\star = \{d\in \cD : r^\top d = R^\star\}\subseteq\cD$. 
This shows that state-action natural gradients not only optimize the reward but produce the policy that induces a state-action distribution with maximal entropy with respect to the initial state-action distribution $d^{\pi_0}$. 
This characterizes the \emph{implicit bias} of state-action natural policy gradients. 
Prior, the implicit bias of a natural actor-critic method has been analyzed by \cite{hu2021actor}, where they provided an $O(\log k)$ bound on the optimal policy with maximal (weighted) entropy over the states. 
Note that as this bound grows with the number of iterations $k$ it can't identify the limiting policy. 
Further, by using the reformulation as a Hessian gradient flow from~\cite{muller2023geometry} and results from convex optimization~\cite{alvarez2004hessian} convergence towards the (generalized) maximal entropy policy for Kakade's natural policy gradient has been established in~\cite{muller2024essentially}. 
\end{remark}

\begin{remark}[Comparison to previous rates]\label{rem:strict-improvements}
In \Cref{thm:linearConvergenceTabular} we provide linear convergence with exponent $\Delta$ and have discussed in \Cref{rem:comparison} this exponent improves on previously established $O(e^{}-\delta t)$ guarantee in \cite{weed2018explicit, suarez2023perspectives}. 
To see this, we note that $d^{\pi_1}, d^{\pi_2}\in\mathscrsfs{D}$ are neighboring vertices if and only if $\pi_1, \pi_2\in\Delta_\mathbb A^\mathbb S$ are neighbors~\cite{muller2022pomdps}. 
Two policies are neighboring if and only if they are deterministic and agree on all but one state. 
Hence, if we consider an MDP with more than one state, there is at least one state $s\in\mathbb S$ such that $\pi_1(a|s) = \pi_2(a|s)=1$ for some $a\in\mathbb A$ and therefore $d^{\pi_1}(s,a) = d^{\pi_1}(s)>0$ and $d^{\pi_2}(s,a) = d^{\pi_2}(s)>0$. 
Hence, $d^{\pi_1}, d^{\pi_2}\in\mathscrsfs{D}$ do not have disjoint support and as elaborated in \Cref{rem:comparison} this implies $\delta<\Delta$. 
Overall, this shows that for an exploratory MDP with more than one state, we have $\delta <\Delta$, meaning that our convergence rate improves upon~\cite{weed2018explicit, suarez2023perspectives}. 
\end{remark}

\subsection{Computational examples}\label{subsec:computational-example}
We use an example from \cite{kakade2001natural,bagnell2003covariant,morimura2008new} of an MDP with two states $s_1, s_2$ and two actions $a_1, a_2$, with the transitions and instantaneous rewards shown in \Cref{fig:kakadeExample}.  
We make our code
available under~\url{https://github.com/muellerjohannes/fisher-rao-GFs-LPs}. 

\begin{figure}[h!]
    \centering
    \begin{tikzpicture}
    \node[shape=circle, draw=black, minimum size=1cm] (A) at (0,0) {\(s_1\)};
    \node[shape=circle,draw=black,minimum size=1cm] (B) at (3,0) {\(s_2\)};
    \path [->] (B) [bend left=20] edge node[below] {\(a_2\)} (A);
    \path [->] (A) [bend left=20] edge node[above] {\(a_2\)} (B);
    \path [->] (A) [loop left] edge node[left] {\(r=+1,\; a_1\)} (A);
    \path [->] (B) [loop right] edge node[right] {\(a_1,\; r=+2\)} (B);
\end{tikzpicture}
\vspace{-.5cm}
\caption{Transition graph and reward of the MDP example.}
\label{fig:kakadeExample}
\end{figure}
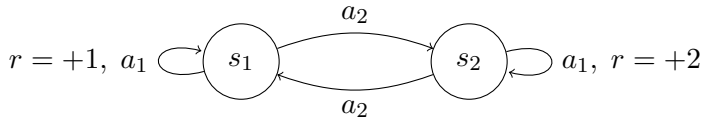

We adopt the initial distribution $\mu(s_1)=0.8, \mu(s_2) = 0.2$ and work with a discount factor of $\gamma = 0.9$. 
We can explicitly compute the rewards of the four deterministic policies to be $R_1=0.98$, $R_2 = 1.2$, $R_3=1.84$ and $R_4=0$, 
and this way determine the optimal policy. 
Consequently, we can compute the exponent $\Delta$ given in \Cref{thm:linearConvergenceTabular} to be $\Delta=0.8$. 
In contrast, the exponent $\delta$ given in~\cite{weed2018explicit, suarez2023perspectives} is $\delta=0.64$. 
In \Cref{rem:comparison} we observed that $\delta\le \Delta$, and this now provides an explicit example where $\delta<\Delta$. 
Finally, we compute the constant $\Delta_{\operatorname{K}}=0.8$ that describes the exponent in the convergence rate of Morimura's natural policy gradient~\cite{khodadadian2022linear}. 

To illustrate our theoretical findings, we run both state-action natural gradients as well as Kakade's natural policy gradient applied to a tabular soft-max parametrization for $30$ random initializations. 
In order to prevent a blow-up of the parameters we use the update rule 
\begin{equation}
    \theta_{k+1} = \theta_k + \eta \cdot G(\theta_k)^+\nabla R(\theta_k),
\end{equation}
with stepsize $\eta>0$, where we choose $\eta = 10^{-2}$ in our experiments. 
Intuitively, we expect $\theta_k \approx \tilde{\theta}_{\eta k}$ if $(\tilde{\theta}_t)_{t\ge0}$ solves the natural policy gradient flow. 

\subsubsection{A first example with tightness} 
\Cref{fig:originalExample} plots the suboptimality gap $R^\star - R(\theta_k)$ as well as the KL-divergence $\KL (d^\star, d_{\theta_k})$ for the two different natural policy gradient methods and the same $30$ random initializations. 
Additionally, the gray dashed line indicates the exponential decay rate $O(e^{-\Delta \eta k}) = O(e^{-\Delta_{\operatorname{K}} \eta k})$ guaranteed by \Cref{thm:linearConvergenceTabular} and by~\cite{khodadadian2022linear}, respectively. 
We see that for all trajectories both the suboptimality gap $R^\star-R(\theta_k)$ as well as the KL-divergence to the optimal state-action distribution $\KL (d^\star, d_{\theta_k})$ decay at this guaranteed rate for both the state-action and Kakade's natural policy gradient method. 
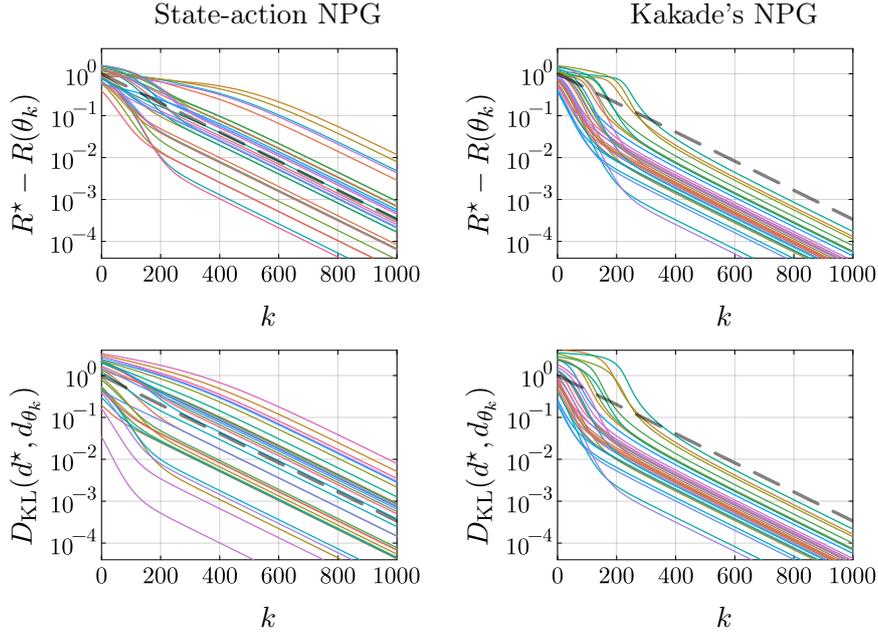
\begin{figure}
    \centering
    \begin{tikzpicture}[scale=1.]
        \node[inner sep=0pt] (r1) at (0,0)
        {\includegraphics[width=5cm]{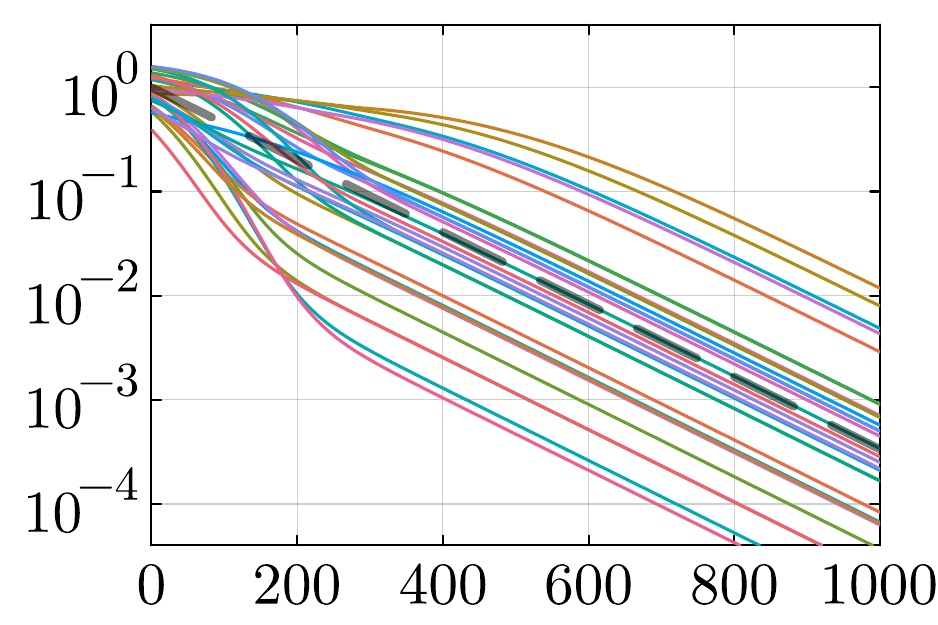}};
    \node[inner sep=0pt] (r2) at (6,0)
        {\includegraphics[width=5cm]{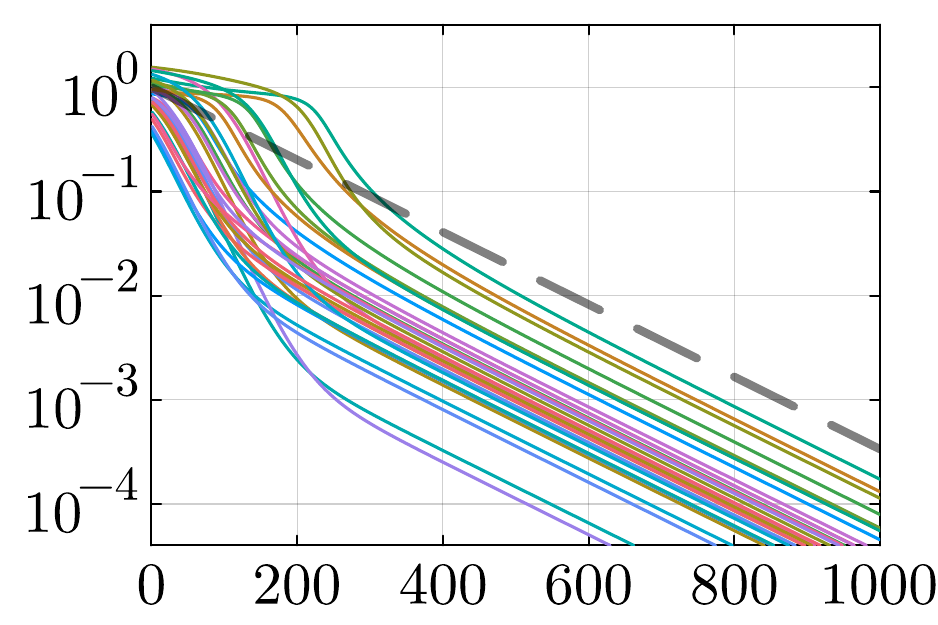}}; 
    \node[inner sep=0pt] (r1) at (0,-4)
        {\includegraphics[width=5cm]{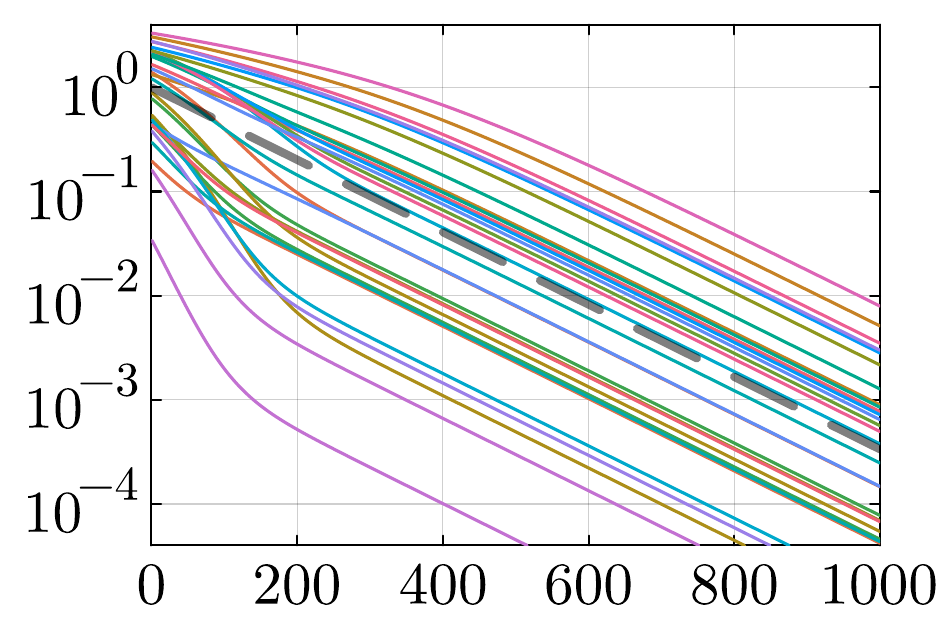}};
    \node[inner sep=0pt] (r2) at (6,-4)
        {\includegraphics[width=5cm]{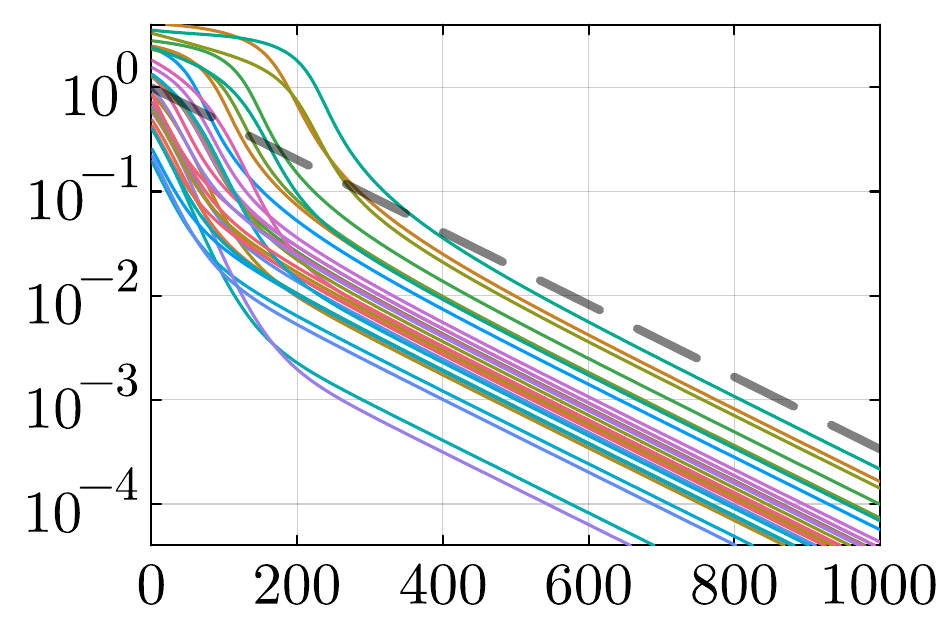}};  
    \node[inner sep=0pt] (l1) at (0.5,2) {State-action NPG}; 
    \node[inner sep=0pt] (l2) at (6.5,2) {Kakade's NPG}; 
    \node[inner sep=0pt] (l1) at (0.5,-2) {$k$}; 
    \node[inner sep=0pt] (l1) at (6.5,-2) {$k$}; 
    \node[inner sep=0pt] (l1) at (0.5,-6) {$k$}; 
    \node[inner sep=0pt] (l1) at (6.5,-6) {$k$}; 
    \node[inner sep=0pt, rotate=90] (l1) at (-2.7,0) {$R^\star - R(\theta_k)$};
    \node[inner sep=0pt, rotate=90] (l1) at (3.3,0) {$R^\star - R(\theta_k)$}; 
    \node[inner sep=0pt, rotate=90] (l1) at (-2.7,-4)  {$\KL (d^\star, d_{\theta_k})$}; 
    \node[inner sep=0pt, rotate=90] (l1) at (3.3,-4) {$\KL (d^\star, d_{\theta_k})$}; 
    \end{tikzpicture}
    \caption{Shown are the suboptimality gap $R^\star - R(\theta_t)$ (top row) and the KL-divergence $\KL (d^\star, d_t)$ (bottom row) for the state-action NPG (left column) and Kakade's NPG (right column) plotted in a logarithmic scale, along with the predicted exponential decay $e^{-\Delta \eta k} = e^{-\Delta_{\operatorname{K}} \eta k}$ (dashed line), see \Cref{thm:linearConvergenceTabular} and \cite{khodadadian2022linear} for state-action and Kakade's NPG, respectively.} 
    \label{fig:originalExample}
\end{figure}
    
\subsubsection{A second example and non-tightness} 
We complement our computational example by studying the same Markov decision process from \Cref{fig:kakadeExample} but changing the reward vector according to $r(s_1,a_2)=3$. 
As above we can compute the three constants $\delta$, $\Delta$ and $\Delta_{\operatorname{K}}$, and obtain $\delta \approx 0.5326$, $\Delta \approx 0.5789$ and $\Delta_{\operatorname{K}} = 1.1$. 
Here again $\delta<\Delta$ as it is always guaranteed under the \Cref{s_asu:exploration}, see \Cref{rem:comparison}. 
Further, in this example, we have $\Delta<\Delta_{\operatorname{K}}$. 
We conduct the same experiment as before and report the findings in \Cref{fig:gapExample}. 
In the plots concerning the state-action natural policy gradient, we plot both the guaranteed decay $O(e^{-\Delta \eta k})$ (gray dashed line) and the decay $O(e^{-\Delta_{\operatorname{K}} \eta k})$ guaranteed for Kakade's natural policy gradient (gray dotted line). 
We see that both methods exhibit the convergence rate $O(e^{-\Delta_{\operatorname{K}} \eta k})$. 
In particular, this indicates that our convergence analysis of the Fisher-Rao gradient, although improving on known results, is still not tight for general problems. 

\begin{figure}
    \centering
    \begin{tikzpicture}[scale=1.]
        \node[inner sep=0pt] (r1) at (0,0)
        {\includegraphics[width=5cm]{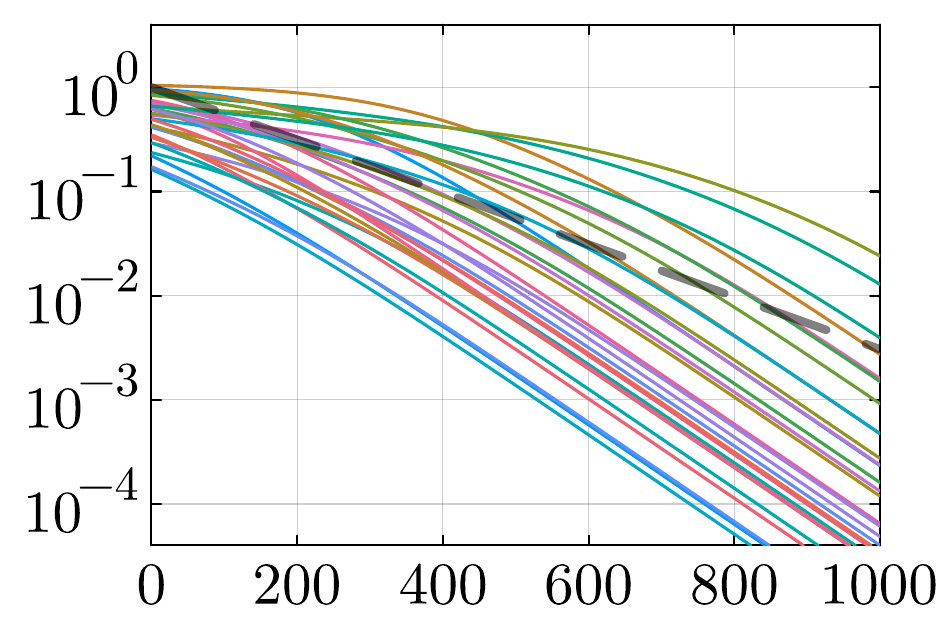}};
    \node[inner sep=0pt] (r2) at (6,0)
        {\includegraphics[width=5cm]{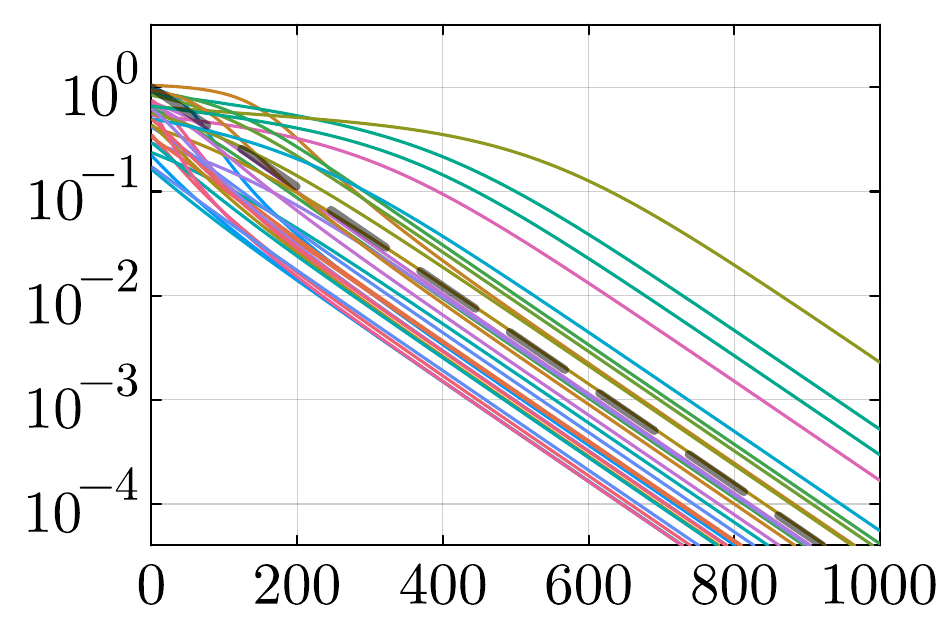}}; 
    \node[inner sep=0pt] (r1) at (0,-4)
        {\includegraphics[width=5cm]{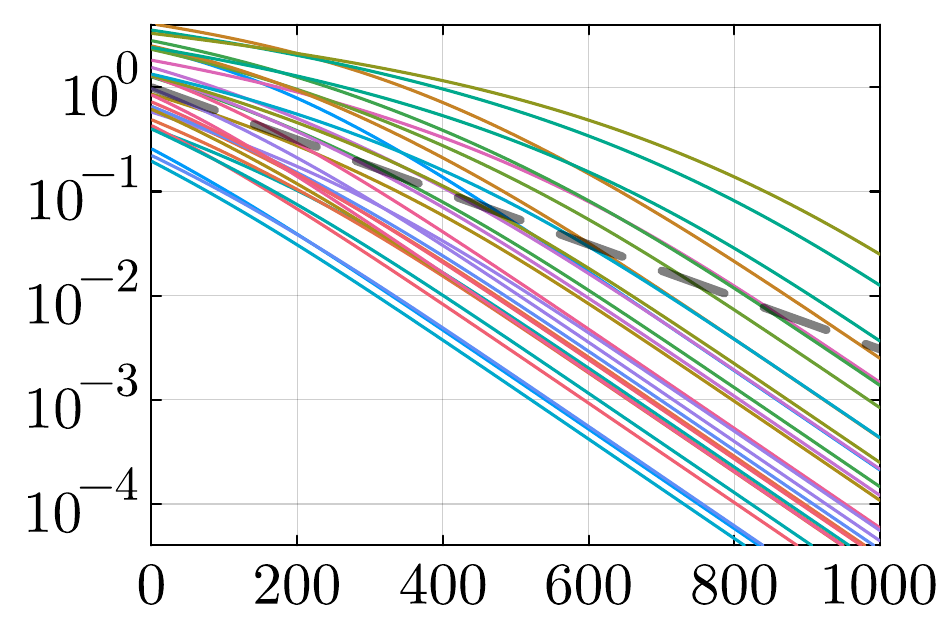}};
    \node[inner sep=0pt] (r2) at (6,-4)
        {\includegraphics[width=5cm]{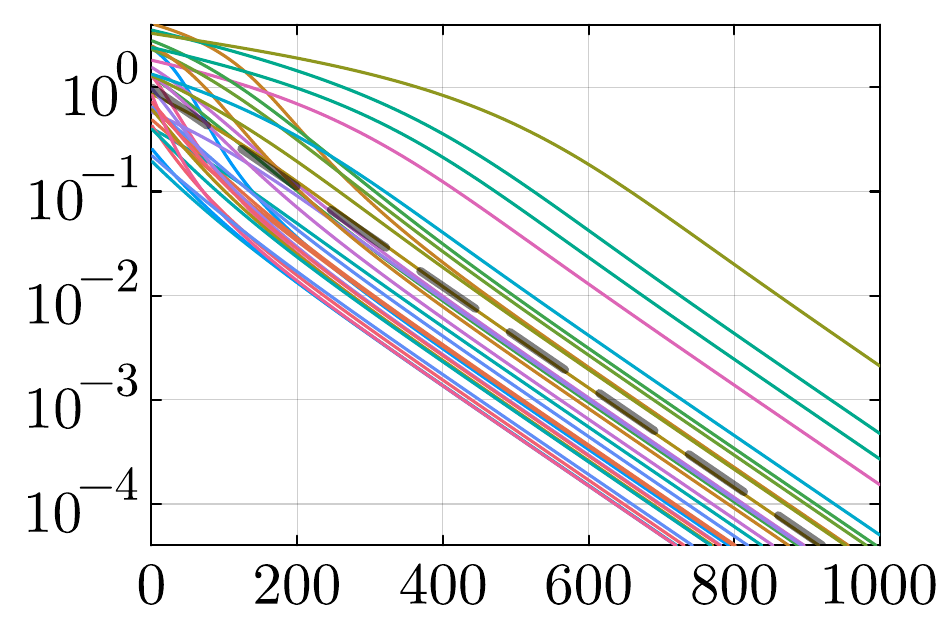}};  
    \node[inner sep=0pt] (l1) at (0.5,2) {State-action NPG}; 
    \node[inner sep=0pt] (l2) at (6.5,2) {Kakade's NPG}; 
    \node[inner sep=0pt] (l1) at (0.5,-2) {$k$}; 
    \node[inner sep=0pt] (l1) at (6.5,-2) {$k$}; 
    \node[inner sep=0pt] (l1) at (0.5,-6) {$k$}; 
    \node[inner sep=0pt] (l1) at (6.5,-6) {$k$}; 
    \node[inner sep=0pt, rotate=90] (l1) at (-2.7,0) {$R^\star - R(\theta_k)$};
    \node[inner sep=0pt, rotate=90] (l1) at (3.3,0) {$R^\star - R(\theta_k)$}; 
    \node[inner sep=0pt, rotate=90] (l1) at (-2.7,-4)  {$\KL (d^\star, d_{\theta_k})$}; 
    \node[inner sep=0pt, rotate=90] (l1) at (3.3,-4) {$\KL (d^\star, d_{\theta_k})$}; 
    \end{tikzpicture}
    \caption{Shown are the suboptimality $R^\star - R(\theta_k)$ (top) and  KL-divergence $\KL (d^\star, d_{\theta_k})$ (bottom) for the state-action NPG (left) and Kakade's NPG (right); shown are also the guaranteed exponential decay rates $e^{-\Delta \eta k}$ for the state-action NPG (dashed line) and $e^{-\Delta_{\operatorname{K}} \eta k}$ for Kakade's NPG (dotted line). Although the guarantees are different, both methods exhibit the same fast decay rate.}
    \label{fig:gapExample}
\end{figure}
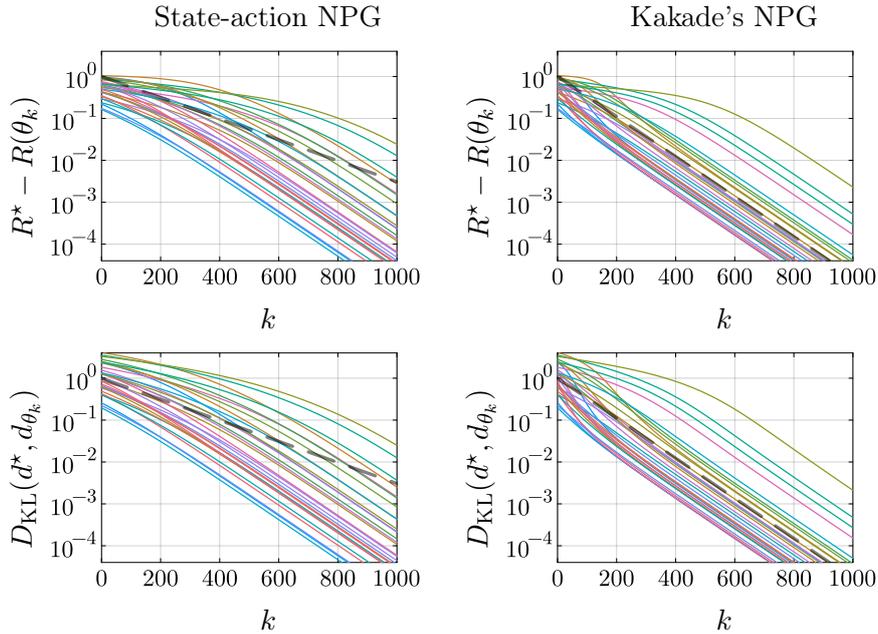

\section{Conclusion and Outlook}
\label{sec:conclusions}
We study Fisher-Rao gradient flows of linear programs and show they converge linearly with an exponent that depends on the geometry of the linear program. 
This yields an estimate on the error introduced by entropic regularization of the linear program, which improves existing guarantees. 
We extend this analysis to natural gradient flows for general parametrized measure models and show they converge at a sublinear rate $O(\frac1t)$ up to an approximation error and mismatch of the trajectory to the solution measure in the $\chi^2$-divergence. 
In particular, our results yield $O(\frac1t)$ convergence of state-action natural policy gradients without regularization under function approximation and linear convergence of state-action natural policy gradients for general tabular parametrizations. 
Finally, we provide computational examples illustrating our results. 

Our results improve previous results, but some further improvements may be possible. 
In particular, we use the best global constant $\Delta>0$ for which the estimate~\eqref{eq:gapVSNorm-2} holds for all $\mu\in P$. However, if one can improve this constant along the trajectory $(\mu_t)_{t\ge0}$ this would directly imply an improvement of the convergence rate. A natural way to approach this is to characterize the direction from which the flow $(\mu_t)_{t\ge0}$ is approaching the global optimizer $\mu^\star$. 
Another interesting direction for future work is to study the statistical complexity of the state-action natural policy gradients. 
Finally, it could be explored whether our convergence results can be used in order to modify the cost to achieve a faster convergence without changing the optimizer, which is known as reward shaping in the context of reinforcement learning. 

\section*{Acknowledgments}
The project originated when JM was a PhD student at the International Max Planck Research School \emph{Mathematics in the Sciences} at MPI MiS with additional support from the \emph{Evangelisches Studienwerk Villigst e.V.}. 
SC and JM acknowledge funding by the Deutsche Forschungsgemeinschaft (DFG, German Research Foundation) under the project number 442047500 through the Collaborative Research Center \emph{Sparsity and Singular Structures} (SFB 1481).
GM has been supported in part by NSF CAREER 2145630, NSF 2212520, DFG SPP 2298 project 464109215, ERC 757983, and BMBF in DAAD project 57616814. 
\appendix

 \section{Auxiliary results}

\begin{lemma}\label{app:lem:cone}
    Consider a polytope $P\subseteq\R^\X$, a face $F\subseteq P$ and consider the cone 
    \begin{align*}
        C\coloneqq \operatorname{cone}\Big\{ \nu-\mu : \mu\in\operatorname{vert}(F), \nu\in N(\mu)\setminus F \Big\}, 
    \end{align*}
    which is generated by the edges pointing out of $F$. 
    Then we have $P\subseteq F+C$. 
\end{lemma}
\begin{proof}
This is a generalization of~\cite[Lemma 3.6]{ziegler2012lectures}, which covers the case that $F$ consists of a single vertex. 
We will show that 
\begin{align*}
    F+C \supseteq\tilde P \coloneqq \bigcap \Big\{ H\subseteq \R^\X : H \text{ is a halfspace }, P\subseteq H, H\cap F\ne\emptyset \Big\} \supseteq P, 
\end{align*}
for which we pick an element $u\in \tilde P$. 
Consider a hyperplane $H = \{\mu:a^\top \mu = \alpha\}$ separating $F$ and $\operatorname{vert}(P)\setminus F$ and consider the \emph{face figure} $P/F\coloneqq P\cap H$, which is a polytope. 
Now, we consider a translation $\tilde H = \{ \mu : a^\top\mu= \beta \}$ of $H$,  such that $u\in \tilde H$. 
Now we have 
\begin{align*}
    \tilde P = \operatorname{conv}\left\{ \mu + \frac{a^\top \mu - \beta}{a^\top \mu - a^\top \nu} \cdot (\nu-\mu) : \mu\in\operatorname{vert}(F), \nu\in N(\mu)\setminus F \right\},
\end{align*}
see~\cite[Proposition 2.30]{zieglerdiscrete}. 
Hence, we can choose convex weights $\lambda_i$ such that 
\begin{align*}
    u = \sum_{i} \lambda_i (\mu_i + \alpha_i (\nu_i-\mu_i)) =  \sum_{i} \lambda_i \mu_i + \sum_{i} \alpha_i \lambda_i (\nu_i-\mu_i) \in F + C,
\end{align*}
where $\mu_i\in\operatorname{vert}(F), \nu_i \in N(\mu_i)\setminus F$.
\end{proof}

\setcounter{theorem}{11}
\setcounter{section}{3}
\renewcommand\thesection{\arabic{section}}

\renewcommand\thesection{\Alph{section}}
\setcounter{section}{1}
\setcounter{theorem}{1}

\begin{lemma}[Information projections have maximal support]\label{app:lem:support}
    Consider a polytope $P=\Delta_\mathbb X \cap L$ for an affine space $L$ and a face $F$ of $P$. 
    Further, let $\hat\mu\in F$ be the information projection of $\mu\in \operatorname{int}(P)$ to $F$, then $\hat\mu\in\operatorname{int}(F)$. 
\end{lemma}
\begin{proof}
    Note that $\hat\mu\in F$ is characterized by $D_{\operatorname{KL}}(\hat\mu, \mu) = \min_{\mu'\in F} D_{\operatorname{KL}}(\mu', \mu)$. 
    Assume that $\hat\mu\in\partial F$, then $\mu_{x_0}=0$ for some $x_0\in\X$. 
    Consider now $v\in\R^\X$ such that $\hat\mu+tv\in\operatorname{int}(F)$ for $t>0$ small enough, then surely $v_{x_0}>0$. 
    By convexity of the KL-divergence, we have $D_{\operatorname{KL}}(\hat\mu, \mu) \ge D_{\operatorname{KL}}(\hat\mu+tv, \mu) + t \partial_t D_{\operatorname{KL}}(\hat\mu+tv, \mu)$, where $\partial_t D_{\operatorname{KL}}(\hat\mu+tv, \mu)\to-\infty$ for $t\to0$. 
    This shows $D_{\operatorname{KL}}(\hat\mu, \mu) > D_{\operatorname{KL}}(\hat\mu+tv, \mu)$ for $t$ small enough contradicting that $\hat\mu$ is the information projection of $\mu$. 
\end{proof}

\bibliographystyle{plain} 
\bibliography{references}

\end{document}